\documentclass[oneside,11pt]{article}
\usepackage{caption}
\usepackage{blindtext}

\usepackage[abbrvbib, preprint]{jcustom}

\usepackage{lastpage}

\usepackage[utf8]{inputenc} 
\usepackage[T1]{fontenc}           
\usepackage{booktabs} 
\usepackage{hyperref}
\usepackage{enumitem}
\usepackage{algorithm}
\usepackage{algorithmicx}
\usepackage{algpseudocode}

\hypersetup{
    colorlinks=true,
    allcolors=blue
}

\usepackage{amsfonts,amsmath,amsthm,amssymb}       
\usepackage{nicefrac}       
\usepackage{microtype}      
\usepackage{xcolor}         
\usepackage{pifont}
\usepackage{tikz}
\usepackage{subcaption}

\usetikzlibrary{arrows.meta}
\usepackage{wrapfig}
\usepackage[presets={vec-cev}]{letterswitharrows}
\usepackage{cleveref}
\AddToHook{cmd/appendix/before}{
    \crefalias{section}{appendix}%
    \crefalias{subsection}{appendix}
}
\usepackage{graphicx}

\newtheorem{proposition}{Proposition}
\newtheorem{lemma}{Lemma}
\theoremstyle{remark}
\newtheorem{remark}{Remark}

\definecolor{src}{RGB}{76,114,176}
\definecolor{tgt}{RGB}{221,132,82}

\tikzset{
  traj/.style={black!75, line width=0.9pt, -{Latex[length=1.8mm]}},
  noisytraj/.style={black!70, line width=0.8pt, -{Latex[length=1.6mm]}, opacity=0.75},
  margsrc/.style={fill=src!12, draw=src!55, line width=0.5pt},
  margtgt/.style={fill=tgt!12, draw=tgt!55, line width=0.5pt},
  sp/.style={fill=src, draw=src},
  tp/.style={fill=tgt, draw=tgt},
}

\theoremstyle{thmstyleone}
\newtheorem{theorem}{Theorem}[section]

\theoremstyle{thmstylethree}
\newtheorem{assumption}[theorem]{Assumption}

\theoremstyle{thmstyletwo}

\crefname{assumption}{assumption}{assumptions}
\Crefname{assumption}{Assumption}{Assumptions}

\newcommand{\R}{\mathbb{R}}
\newcommand{\ip}[2]{\langle #1,#2\rangle}
\newcommand{\abs}[1]{\lvert #1\rvert}
\newcommand{\norm}[1]{\lVert #1\rVert}
\newcommand{\mcE}{\mathcal{E}}
\newcommand{\mcR}{\mathcal{R}}
\newcommand{\mcJ}{\mathcal{J}}
\newcommand{\mcS}{\mathcal{S}}
\newcommand*\dif{\mathop{}\!\mathrm{d}}
\newcommand{\dd}{\mathrm{d}}

\title{Dirac--Frenkel dynamics with inertia for nonlinearly parametrized solutions of evolution problems}

\usepackage{sidecap}
\sidecaptionvpos{figure}{c}  

\author{%
  \name Matteo Raviola \\
  \addr Scientific Computing and Uncertainty Quantification - CADMOS Chair, EPFL\\
  \AND
  Benjamin Peherstorfer\\
   \addr Courant Institute of Mathematical Sciences, New York University\\
}

\newcommand{\tc}[1]{#1}

\begin{document}

\maketitle
\begin{abstract}
Even when Dirac--Frenkel dynamics determine a well-defined evolution in function space, the corresponding parameter dynamics can be non-unique or ill-conditioned for redundant nonlinear parametrizations, such as typical neural networks or mixture models. We propose to add inertia to the Dirac--Frenkel dynamics and show that this allows useful parameter velocity information to persist from the past trajectory in directions that are weakly informed, while well-informed parameter velocity directions continue to follow the Dirac--Frenkel dynamics. We prove that the inertial formulation yields well-posed parameter dynamics and provide a posteriori error bounds. After time discretization, the method requires the solution of the same type of regularized linear least-squares problem as standard Dirac--Frenkel dynamics, but with the previous velocity appearing as an anchor. Numerical experiments demonstrate the increased robustness obtained with inertia.
\end{abstract}

\section{Introduction}
Let
\begin{equation}
      \dot u(t)=F(u(t)),
      \qquad u(0)=u^0,
      \label{eq:ode_intro}
\end{equation}
be an evolution problem on a Hilbert space $H$. We approximate the solution by a
nonlinear parametrization
\begin{equation}
      \hat u(t)=\Phi(\theta(t)),
      \qquad \theta(t)\in \Theta,
      \label{eq:param_intro}
\end{equation}
where $\Theta$ is a $p$-dimensional Hilbert space and $\Phi:\Theta\to H$. For example, $\Phi$ could be a neural network and $\theta(t)$ the weights of the neural network. The Dirac--Frenkel variational principle determines a parameter velocity $\dot{\theta}(t)$ by asking
that the tangent vector $J(\theta(t))\dot\theta(t)$, with $J(\theta(t))w:=D\Phi(\theta(t))[w]$, be the best instantaneous
approximation of the vector field $F(\Phi(\theta(t)))$ in $H$.
This point of view is classical in quantum dynamics and dynamical low-rank
approximation \cite{Dirac1930,Frenkel1934,Lubich2008,KochLubich2007,Haegeman2011} and it has become increasingly relevant for nonlinear
parametrizations given by neural networks such as Neural Galerkin schemes \cite{bruna_neural_2024} and other techniques \cite{PhysRevE.104.045303,anderson2021evolution,wang2021particle,doi:10.1137/22M1529427,finzi_stable_2023,Feischl2024,lubich2025regularizeddynamicalparametricapproximation,10.1051proc202581002,KAST2024112986,24OTDDTO,dahmen2025expansivenaturalneuralgradient,bon2025stablenonlineardynamicalapproximation,hesthaven2026nonlinearmodelreductiontransportdominated}.

A central difficulty is that the Dirac--Frenkel principle determines the function-space tangent vector $J(\theta(t))\dot\theta(t)$, but the corresponding parameter velocity $\dot\theta(t)$ need not be unique or well conditioned. For redundant parametrizations, such as neural networks or Gaussian mixtures, the Jacobian $J(\theta(t))$ may have a nontrivial kernel, so different parameter velocities $\dot{\theta}(t)$ produce the same function-space velocity $J(\theta(t))\dot{\theta}(t)$. Even if the Jacobian is full rank, small singular values of $J(\theta(t))$ make the instantaneous problem of fitting $F(\Phi(\theta(t)))$ by a tangent vector $J(\theta(t))\dot\theta(t)$ ill-conditioned. Components of the desired tangent vector that can be matched only through small singular directions of $J(\theta(t))$ then require large, unstable coefficients in the parameter velocity. Thus, while the evolution in function space may be well determined,  the induced dynamics in parameter space are non-unique or ill-conditioned. This phenomenon is called tangent space collapse \cite{24OTDDTO} or matrix singularity issue \cite{KAY1989165} and is widely recognized; see, e.g., \cite{9073ba01-c8c8-3f30-b15c-e4b52a44e9da,berman2023randomized,finzi_stable_2023,Feischl2024,pmlr-v235-chen24ad}.

One remedy is to regularize the instantaneous Dirac--Frenkel problem so that a unique parameter velocity is selected.  The work \cite{Feischl2024} provides a detailed analysis of Tikhonov-regularized Dirac--Frenkel dynamics. Tikhonov regularization and related truncated singular-value rules select a velocity by damping or removing directions that are weakly informed by the Jacobian $J(\theta(t))$ at the current parameter $\theta(t)$. Another line of work builds on randomization to use partial or compressed information from the current Jacobian. For example, one may evaluate the Jacobian action $J(\theta(t))v$ only
at randomly sampled spatial points, or apply a random sketch matrix before solving for the velocity \cite{berman2023randomized,lam2024randomizedlowrankrungekuttamethods,dong2025randomizedtimesteppingnonlinearly,CARREL2026114421}.
These randomized methods can reduce the cost of forming and solving the least-squares problem, and in some cases improve conditioning.
Other works propose to retrain the neural network \cite{finzi_stable_2023} or use different dynamics from Dirac--Frenkel dynamics \cite{kvaal2023need,pmlr-v202-chen23af,24OTDDTO,pmlr-v235-chen24ad}. In most of these approaches, the velocity is still based on the current state and Jacobian only, through a regularizer, truncation, randomization rule, or modified local problem.
Closest to our work is the Dirac--Frenkel--Onsager approach of \cite{RP26DFO}, which interprets the non-uniqueness as gauge freedom. That method preserves the instantaneous Dirac--Frenkel residual minimization and uses an additional Onsager-type variable to select how the parameters move in nullspace directions.

The present paper takes a different route. We introduce Dirac--Frenkel dynamics with inertia (DFI), which replace the instantaneous selection of a parameter velocity by an evolution equation for the velocity itself. The velocity is driven toward satisfying the current Dirac--Frenkel condition, but it is not forced to forget its past at every instant (time step).
This distinguishes DFI from the Dirac--Frenkel--Onsager approach of \cite{RP26DFO}, which preserves the instantaneous Dirac--Frenkel residual minimization and uses history only to resolve the gauge freedom in nullspace directions. DFI instead applies inertia to the full parameter velocity, analogous to how momentum is used in optimization \cite{POLYAK19641,doi:10.1142/S0219199700000025,pmlr-v202-muller23b,GOLDSHLAGER2024113351,guzman-cordero2025improving,nouy2026naturalgradientdescentmomentum}. As a result, DFI behaves like the usual Dirac--Frenkel dynamics in directions that are well informed by the current Jacobian, while dynamics given by the inertia dominate in directions that are weakly informed, whose instantaneous correction is strongly regularized, or directions seen only through sketched information. \tc{We show that, once the initial condition is fixed, DFI yields a well-defined evolution in parameter space. Its use of past trajectory information provides a mechanism for robustness when the current least-squares problem is weakly informative.}
After time discretization, a DFI time step leads to a previous-velocity-anchored least-squares solve, which requires the same type of regularized least-squares solve as standard regularized Dirac--Frenkel methods, with the previous velocity appearing as an anchor.
      \tc{We provide a posteriori error bounds for the time-discrete scheme. Furthermore, we argue that temporal coherence of the parameter velocities can make DFI compatible with stronger effective regularization that leads to potentially cheaper instantaneous least-squares solves with, e.g., more aggressively sketched solvers. The experiments indeed demonstrate that DFI remains accurate in strongly regularized and sketched regimes.}

\section{Preliminaries and problem formulation}
\label{sec:setting}
We briefly recap the Dirac--Frenkel variational principle and discuss why parameter dynamics can be underdetermined. 

\subsection{Dirac--Frenkel variational principle}
We work with a real Hilbert state space $H$ and a $p$-dimensional real Hilbert parameter space $\Theta$.
Let $\Phi:\Theta\to H$ be of class $C^2$, let $F:H\to H$ be the vector field of \eqref{eq:ode_intro}, and write $\hat u(t)=\Phi(\theta(t))$ along a trajectory.
In the following, when discussing the instantaneous Dirac--Frenkel problem at a fixed parameter $\theta$, we suppress the time argument and, for $\theta\in \Theta$, set
        \begin{equation}
          \hat u=\Phi(\theta),
          \qquad J(\theta)w:=D\Phi(\theta)[w]\quad(w\in \Theta),
          \qquad f(\theta):=F(\hat u)\in H.
          \label{eq:Jf_def}
        \end{equation}
Note that since $\Theta$ is a finite-dimensional Hilbert vector space, we identify each tangent space
with $\Theta$. Thus parameter velocities  are in  $\Theta$.

At fixed $\theta$, define the function-space Dirac--Frenkel defect by
        \begin{equation}
          \mathcal{E}(\dot{\hat u})
          :=\frac12\norm{\dot{\hat u} - F(\hat u)}_H^2.
          \label{eq:DF_function_level}
        \end{equation}
        The Dirac--Frenkel principle selects a velocity by minimizing \eqref{eq:DF_function_level} over all velocities $\dot{\hat u}$ in the range $\mathrm{Im} (D\Phi(\theta))$, which is the tangent space at the representative $\theta$ under suitable regularity assumptions. 
At the parameter level, using the chain rule $\dot{\hat u}=J(\theta)\dot \theta$, this corresponds to selecting the parameter velocity $\dot \theta$ 
        \begin{equation}
          \dot \theta \in \arg\min_{v\in \Theta} \mathcal{E}_\theta(v)
          \label{eq:DF_parameter_level}
        \end{equation}
        by minimizing the parameter-space Dirac--Frenkel defect
        \begin{equation}\label{eq:DFDefect}
        \mathcal{E}_\theta(v)
    :=\mathcal{E}(J(\theta)v)
          =\frac12\norm{J(\theta)v-f(\theta)}_H^2\,.
        \end{equation}

\subsection{Non-unique or ill-conditioned parameter dynamics}\label{sec:ProblemFormulation}
We now discuss that the Dirac--Frenkel dynamics impose dynamics in function space that can lead to underdetermined parameter dynamics. 
\subsubsection{Non-unique parameter dynamics}
For $\theta\in \Theta$, let
        \begin{equation}
          \mcS_0(\theta):=\arg\min_{v\in \Theta}\mathcal{E}_\theta(v).
          \label{eq:S0_def}
        \end{equation}
        All elements of $\mcS_0(\theta)$ are (unregularized) Dirac--Frenkel velocities at $\theta$ that minimize the defect \eqref{eq:DFDefect}. If $\bar v(\theta)\in\mcS_0(\theta)$ is any fixed minimizer, then
        \begin{equation}
          \mcS_0(\theta)=\bar v(\theta)+\ker J(\theta)=\bar v(\theta)+\ker\bigl(J(\theta)^*J(\theta)\bigr).
          \label{eq:gauge_freedom}
        \end{equation}
        Thus, the ambiguity lies  in  directions of the velocity that do not change the tangent vector in $H$.

  \subsubsection{Tikhonov-regularized Dirac--Frenkel dynamics} One way to define a unique velocity is to replace the unregularized Dirac--Frenkel dynamics with their Tikhonov-regularized counterpart. The work \cite{Feischl2024}  introduces the Tikhonov-regularized Dirac--Frenkel functional
        \begin{equation}
          \mcE_\theta^\varepsilon(v) :=\frac12\norm{J(\theta)v-f(\theta)}_H^2+\frac{\varepsilon^2}{2}\norm{v}_\Theta^2,
          \qquad v\in \Theta,
          \label{eq:Etheta}
        \end{equation}
        with regularization parameter $\varepsilon>0$, and sets $\dot{\theta} = \bar{v}_{\varepsilon}(\theta)$ with $\bar{v}_{\varepsilon}(\theta)$ given by 
        \begin{equation}
          \bar{v}_{\varepsilon}(\theta) =\arg\min_{v\in \Theta}\mcE_{\theta}^\varepsilon(v).\label{eq:TDFVelocity}
        \end{equation}
  This yields the regularized Dirac--Frenkel dynamics with 
        \begin{equation}
          \dot{\theta} = \bar{v}_{\varepsilon}(\theta)\,,\qquad \bar{v}_{\varepsilon}(\theta) =M_{\varepsilon}(\theta)^{-1}g(\theta),
          \label{eq:continuous_df_eps}
        \end{equation}
        where
        \begin{equation}
          M_{\varepsilon}(\theta):=J(\theta)^*J(\theta)+\varepsilon^2 I,
          \quad g(\theta):=J(\theta)^*f(\theta)=J(\theta)^*F(\Phi(\theta))\,,
          \label{eq:M_eps_and_force_map}
        \end{equation}
with the adjoint $J(\theta)^*$ of $J(\theta)$ with respect to the inner products of $H$ and $\Theta$. 

\subsubsection{Tikhonov regularization is a  shrinkage rule} To see how Tikhonov regularization acts on parameter directions that are weakly or not at all informed by the Jacobian, freeze $\theta$ and let  $q_i\in\Theta$ be right singular vectors of $J(\theta)$ with singular values $\sigma_i\ge 0$. Let us now consider the velocity in these singular directions by setting
\[
  \bar{v}_i^\varepsilon:=\ip{\bar v_\varepsilon(\theta)}{q_i}_\Theta,
  \qquad
  g_i:=\ip{g(\theta)}{q_i}_\Theta .
\]
The Tikhonov solution \eqref{eq:continuous_df_eps} 
decouples in this singular-vector basis as
\begin{equation}\label{eq:Prelim:ProbForm:InstanSelectRule}
  (\sigma_i^2+\varepsilon^2)\bar{v}_i^\varepsilon=g_i .
\end{equation}
For $\sigma_i>0$, with the left-singular vector $J(\theta)q_i=\sigma_i u_i$, this gives
\begin{equation}\label{eq:Prelim:ProbForm:DampedRuleABC}
  \bar{v}_i^\varepsilon
  =
  \frac{\sigma_i}{\sigma_i^2+\varepsilon^2}
  \ip{u_i}{f(\theta)}_H
  =
  \frac{\sigma_i^2}{\sigma_i^2+\varepsilon^2}
  \underbrace{\left(
    \frac{1}{\sigma_i}\ip{u_i}{f(\theta)}_H
  \right)}_{\bar{v}_i^0}.
\end{equation}
Thus each positive singular direction is damped relative to the unregularized
pseudoinverse coefficient $\bar{v}_i^0$ by the factor $\sigma_i^2/({\sigma_i^2+\varepsilon^2})$. 
This damping is strongest for small singular values, precisely the directions that are only weakly informed by the current Jacobian. In exact null directions, $q_i\in\ker J(\theta)$, so $\sigma_i=0$, and the forcing coefficient satisfies
\[
  g_i
  =
  \ip{J(\theta)^*f(\theta)}{q_i}_\Theta
  =
  \ip{f(\theta)}{J(\theta)q_i}_H
  =
  0,
\]
so that the scalar Tikhonov equation \eqref{eq:Prelim:ProbForm:InstanSelectRule} reduces to
\begin{equation}\label{eq:Prelim:ProbForm:HelperEps0}
  \varepsilon^2 \bar{v}_i^\varepsilon=0\,,
\end{equation}
because $\varepsilon > 0$. Thus, the Tikhonov-selected velocity has zero orthogonal projection onto $\ker J(\theta)$. 
In this sense, Tikhonov regularization is a static, instantaneous selection rule: it damps every positive singular direction, most strongly the weakly informed directions associated with small singular values, and sets exact nullspace components to zero. 

A truncated pseudoinverse used in, e.g., \cite{bruna_neural_2024} is even more abrupt. Directions whose singular values fall below the truncation threshold are simply ignored and their velocity components are set to zero. 

In summary, both mechanisms---Tikhonov regularization and truncated singular value decomposition (SVD)---avoid the parameter non-uniqueness by suppressing velocity directions that are not sufficiently visible to the current Jacobian. In particular, damping occurs through an instantaneous shrink-or-delete rule so that Tikhonov and truncated SVD regularization are  instantaneous shrinkage rules in parameter space.

\section{Dirac--Frenkel dynamics with inertia}
We now introduce Dirac--Frenkel dynamics with inertia (DFI). Through inertia, the parameter velocity carries information from the past of the trajectory, so motion can persist in parameter directions associated with zero or small singular values of the instantaneous Jacobian.

\subsection{Dirac--Frenkel dynamics with inertia}
\label{sec:dfo}
For fixed $\theta\in \Theta$ and current velocity $\dot \theta \in \Theta$, we choose an acceleration $\ddot \theta \in \Theta$ by an Onsager-type variational principle applied to the unregularized defect energy \eqref{eq:DFDefect}, namely
        \begin{equation}
          \ddot \theta\in\arg\min_{a\in \Theta}\mcR^0_{\theta,\dot\theta}(a):=\frac{\tau^2}{2}\norm{a}_\Theta^2+\ip{\nabla \mathcal{E}_\theta(\dot\theta)}{a}_\Theta,
          \qquad \tau>0\,,
          \label{eq:Onsager_principle_unreg}
        \end{equation}
where $\nabla \mcE_\theta$ denotes the gradient with respect to the Hilbert-space inner product on $\Theta$. 
The above selects an acceleration $\ddot \theta$ that minimizes the sum of two terms: a quadratic penalty on acceleration, which prevents rapid changes in the velocity, and the rate of change of the unregularized Dirac--Frenkel defect energy in the direction of the acceleration, which promotes minimization of the defect energy. 
The parameter $\tau^2$ plays the role of a mass in parameter space, which prevents the parameter velocity from changing abruptly.

 Because $\mcR^0_{\theta,\dot\theta}$ is strictly convex, the minimizer is unique and solves
        \begin{equation}
          \tau^2 \ddot \theta+J(\theta)^*\bigl(J(\theta)\dot\theta-f(\theta)\bigr)=0.
          \label{eq:Onsager_stationary_unreg}
        \end{equation}
   We therefore obtain the  DFI system
        \begin{equation}
          \dot\theta=v,
          \qquad
          \tau^2\dot v+J(\theta)^*J(\theta)v=J(\theta)^*f(\theta).
          \label{eq:continuous_sdfr_unreg}
        \end{equation}

\subsection{Tikhonov-regularized Dirac--Frenkel dynamics with inertia}
\label{sec:dfo_eps}
The same inertial mechanism can be combined with the Tikhonov-regularized Dirac--Frenkel functional \eqref{eq:Etheta}.
 For fixed $\theta\in \Theta$ and current velocity $\dot{\theta}\in \Theta$, we choose an acceleration $\ddot{\theta}\in \Theta$ by Onsager's principle applied to \eqref{eq:Etheta}:
        \begin{equation}
          \ddot{\theta} \in\arg\min_{a\in \Theta}\mcR^{\varepsilon}_{\theta,\dot{\theta}}(a):=\frac{\tau^2}{2}\norm{a}_\Theta^2+\ip{\nabla_{\dot{\theta}}\mcE_\theta^\varepsilon(\dot{\theta})}{a}_\Theta,
          \qquad \tau>0.
          \label{eq:Onsager_principle}
        \end{equation}
  Because $\mcR^{\varepsilon}_{\theta,\dot{\theta}}$ is strictly convex, the minimizer is unique and is characterized by
        \begin{equation}
          \tau^2 \ddot{\theta}+\nabla_{\dot{\theta}}\mcE_\theta^\varepsilon(\dot{\theta})=0.
          \label{eq:Onsager_stationary}
        \end{equation}
 Using
        \begin{equation}
          \nabla\mcE_\theta^\varepsilon(v)=J(\theta)^*\bigl(J(\theta)v-f(\theta)\bigr)+\varepsilon^2 v
          \label{eq:gradEtheta}
        \end{equation}
        and the definitions  \eqref{eq:M_eps_and_force_map} in \eqref{eq:Onsager_stationary}, we obtain the coupled DFI system
        \begin{equation}
          \dot\theta=v,
          \qquad
          \tau^2\dot v+M_{\varepsilon}(\theta) v=g(\theta)\,,
          \label{eq:continuous_dfo}
        \end{equation}
  which can be written in second-order form as 
        \begin{equation}
          \tau^2\ddot\theta+M_{\varepsilon}(\theta)\dot\theta
          =g(\theta).
          \label{eq:continuous_dfo_second_order}
        \end{equation}
 
  For fixed $\theta$, the $v$-equation in \eqref{eq:continuous_dfo} is the gradient flow of $\mcE_\theta^\varepsilon$ in the velocity variable $v$, with mobility $\tau^{-2}I_\Theta$.
       If $\varepsilon=0$, the relaxation is toward the affine set of unregularized Dirac--Frenkel minimizers. In the frozen-$\theta$ dynamics, if $\varepsilon>0$, the gradient flow relaxes toward the unique regularized Dirac--Frenkel velocity. 
       
       \begin{remark}We note that the second-order system \eqref{eq:continuous_dfo_second_order} is analogous to heavy-ball dynamics used in optimization; see, e.g., \cite{POLYAK19641,doi:10.1142/S0219199700000025}. 
       This reflects that the way we use inertia in DFI is analogous to using history information via a momentum variable in optimization, in particular when applied to natural gradient descent \cite{pmlr-v202-muller23b,GOLDSHLAGER2024113351,guzman-cordero2025improving,nouy2026naturalgradientdescentmomentum}. Among these works,  \cite{GOLDSHLAGER2024113351} is closest to our approach. It addresses the optimization problem of fitting neural network wave functions via variational Monte Carlo equipped with momentum. This leads to an anchored least-squares problem for the velocity analogous to our time-discrete formulation in \Cref{sec:euler-dfi}. 
       \end{remark}
       
\subsection{Interpretation of DFI}\label{sec:DFI:Interpretation}
Let us now interpret DFI direction by direction, in direct analogy with the regularization rules discussed in \Cref{sec:ProblemFormulation}.  
Freeze $\theta$ and choose an orthonormal basis $\{q_i\}_{i=1}^p$ of $\Theta$ consisting of eigenvectors of $J(\theta)^*J(\theta)$,
\[
J(\theta)^*J(\theta)q_i=\sigma_i^2q_i,
\]
with singular values $\sigma_i\ge 0$. 
Set
\[
  v_i:=\ip{v}{q_i}_\Theta,
  \qquad
  g_i:=\ip{g(\theta)}{q_i}_\Theta .
\]
Since
\[
  M_\varepsilon(\theta)q_i=(\sigma_i^2+\varepsilon^2)q_i,
\]
the $q_i$-component of \eqref{eq:continuous_dfo} is
\begin{equation}\label{eq:DFI:InterpretEq1}
  \tau^2 \dot v_i+(\sigma_i^2+\varepsilon^2)v_i=g_i ,
\end{equation}
which is in stark contrast to Tikhonov and other instantaneous regularizers that implement an instantaneous algebraic selection such as \eqref{eq:Prelim:ProbForm:InstanSelectRule}. If $\sigma_i^2 + \epsilon^2 > 0$, we can further rewrite \eqref{eq:DFI:InterpretEq1} as
\[
  \tau^2 \dot v_i
  +
  (\sigma_i^2+\varepsilon^2)(v_i-\bar v_i^\varepsilon)
  =
  0 
\]
to make explicit the connection to $\bar{v}_i^{\varepsilon}$ given by \eqref{eq:Prelim:ProbForm:InstanSelectRule}.  Thus, for frozen $\theta$, the Tikhonov coefficient $\bar v_i^\varepsilon$ is not imposed immediately; it is approached via relaxation. The relaxation time in this direction is $\tau^2/(\sigma_i^2+\varepsilon^2)$.
        In directions where $\tau^2 \ll \sigma_i^2+\varepsilon^2$, the velocity rapidly tracks the Tikhonov velocity $\bar{v}_{\varepsilon}$ given by \eqref{eq:continuous_df_eps}, so DFI behaves like the Tikhonov-regularized DF dynamics.
        In directions where $\tau^2 \gg \sigma_i^2+\varepsilon^2$, the relaxation is slow and the velocity is mainly transported by inertia.

For $\sigma_i>0$, the Tikhonov rule leads to the  damped velocity coefficient given in \eqref{eq:Prelim:ProbForm:DampedRuleABC}. 
In DFI, the damped velocity direction  \eqref{eq:Prelim:ProbForm:DampedRuleABC} is only the asymptotic target of the frozen-$\theta$ velocity dynamics. In particular, if the direction is weakly informed, the velocity need not be reset immediately to this small Tikhonov value; it can retain motion from the past trajectory controlled by $\tau$. The contrast between DFI and Tikhonov or truncated regularization is most pronounced in exact null directions.  If $q_i\in\ker J(\theta)$, then 
Tikhonov regularization with $\epsilon > 0$ selects $\bar{v}_i^{\varepsilon} = 0$ (see \eqref{eq:Prelim:ProbForm:HelperEps0}), 
whereas  DFI gives
\[
  \tau^2\dot v_i+\varepsilon^2 v_i=0 .
\]
In the frozen-$\theta$ model, if $\varepsilon=0$, then the velocity in nullspace directions is preserved by DFI. For $\varepsilon>0$, it is not removed instantaneously but decays exponentially. 

Thus, in the frozen-$\theta$ interpretation, DFI permits velocity components in parameter directions that are weakly seen by $J(\theta)$, including exact kernel directions, to persist over a relaxation time scale rather than being removed instantaneously. This makes DFI history-aware because the current residual still informs the velocity in directions resolved by the Jacobian, while components in weakly resolved directions are damped dynamically rather than eliminated by a pointwise algebraic rule.

\subsection{Well-posedness of DFI}
We now turn to the well-posedness of the DFI system \eqref{eq:continuous_dfo}. We first establish local existence and uniqueness, and then global existence under a growth condition on the force map $g$ defined in \eqref{eq:M_eps_and_force_map}.
\begin{proposition}[Local existence and uniqueness of DFI solution]
          \label{prop:local_wellposedness_dfo}
          Recall that $\Phi\in C^2(\Theta;H)$ and assume that $F:H\to H$ is locally Lipschitz.
          Then, for $\tau>0$,  $\varepsilon\ge0$, and  initial datum
          $(\theta^0,v^0)\in \Theta\times \Theta$, there exists a time $T_{\text{max}} \in (0, \infty]$ and a unique solution $(\theta(t), v(t))$ of the DFI system
          \eqref{eq:continuous_dfo} on $[0, T_{\text{max}})$ with $(\theta(0), v(0)) = (\theta^0, v^0)$. Moreover, this solution is maximal in the sense that it cannot be extended to any larger interval $[0, T')$ with $T' > T_{\text{max}}$. 
        \end{proposition}
  \begin{proof} Since $\Theta$ is finite-dimensional, the DFI system is an ordinary differential equation on the finite-dimensional phase space $\Theta\times\Theta$. 
          Since $\Phi\in C^2$, the map $\theta\mapsto J(\theta)=D\Phi(\theta)$ is locally Lipschitz.
          Since $F$ is locally Lipschitz and $\Phi$ is $C^2$, the composition $\theta\mapsto f(\theta)=F(\Phi(\theta))$ is locally Lipschitz as well. Therefore $g(\theta)$ is locally Lipschitz, because
\[
g(\theta_1)-g(\theta_2)
=
\bigl(J(\theta_1)^*-J(\theta_2)^*\bigr)f(\theta_1)
+
J(\theta_2)^*\bigl(f(\theta_1)-f(\theta_2)\bigr).
\]
On each bounded neighborhood in $\Theta$, the maps $J$, $J^*$, and $f$ are locally bounded, and $J$ and $f$ are locally Lipschitz; hence the right-hand side is bounded by a constant times $\|\theta_1-\theta_2\|_\Theta$.

Similarly, 
\[
M_\varepsilon(\theta)=J(\theta)^*J(\theta)+\varepsilon^2I
\]
is locally Lipschitz and thus the map
\[
(\theta,v)\mapsto M_\varepsilon(\theta)v
\]
is locally Lipschitz on $\Theta\times\Theta$. 
          Hence the phase-space vector field
          \begin{equation*}
            G_{\tau,\varepsilon}(\theta,v)
            :=
            \left(
            v,\,
            \tau^{-2}\left(J(\theta)^*f(\theta)-M_\varepsilon(\theta)v\right)
            \right)
          \end{equation*}
          is locally Lipschitz on $\Theta\times \Theta$.
          The classical Cauchy--Lipschitz theorem therefore yields a unique maximal local solution.
        \end{proof}
        
 \begin{proposition}[Global existence of DFI solution]
          \label{prop:global_existence_dfo}
          Assume the hypotheses of Proposition~\ref{prop:local_wellposedness_dfo} and, in addition, that the map $g$ defined in \eqref{eq:M_eps_and_force_map} satisfies
          \begin{equation}
            \norm{g(\theta)}_\Theta\le C_g(1+\norm{\theta}_\Theta),
            \qquad \theta\in \Theta,
            \label{eq:dfo_force_linear_growth}
          \end{equation}
          for some constant $C_g>0$.
          Then every maximal solution of \eqref{eq:continuous_dfo} is global.
        \end{proposition}
 \begin{proof}
          Let $(\theta(t),v(t))$ be a maximal local solution and define
          \begin{equation}\label{eq:ProofGlobalExsist:XDef}
            X(t):=\norm{\theta(t)}_\Theta^2+\tau^2\norm{v(t)}_\Theta^2\,,
          \end{equation}
          which is differentiable because by the Cauchy--Lipschitz theorem, the maximal local solution is continuously differentiable. 
          Along the solution we compute
          \begin{align*}
            \dot X(t)
             & =
            2\ip{\theta(t)}{v(t)}_\Theta
            +2\ip{g(\theta(t))}{v(t)}_\Theta
            -2\norm{J(\theta(t))v(t)}_H^2
            -2\varepsilon^2\norm{v(t)}_\Theta^2                                          \\
             & \leq  2\ip{\theta(t)}{v(t)}_\Theta
            +2\ip{g(\theta(t))}{v(t)}_{\Theta}\,.
            \end{align*}
            Now use Cauchy--Schwarz to obtain
            \[
            \langle \theta(t), v(t)\rangle_{\Theta} \leq \|\theta(t)\|_{\Theta}\|v(t)\|_{\Theta}\,,\qquad \langle g(\theta(t)), v(t)\rangle_{\Theta} \leq \|g(\theta(t))\|_{\Theta}\|v(t)\|_{\Theta}\,.
            \]
            Using Young's inequality $2ab \leq a^2 + b^2$, we obtain
            \begin{align*}
            \dot{X}(t)  & \leq \norm{\theta(t)}_\Theta^2+2\norm{v(t)}_\Theta^2+\norm{g(\theta(t))}_\Theta^2. 
            \end{align*}
            Recall the growth condition \eqref{eq:dfo_force_linear_growth} to obtain
            \[
            \|g(\theta(t))\|_{\Theta}^2 \leq C_g^2(1 + \|\theta(t)\|_{\Theta})^2 \leq 2C_g^2 + 2C_g^2\|\theta(t)\|_{\Theta}^2\,,
            \]
            where we used $(a + b)^2 \leq 2a^2 + 2b^2$ in the last step, and bound $\dot{X}(t)$ as
            \begin{equation}\label{eq:GlobalExtProof:HelperA}
            \dot{X}(t) \leq 2C_g^2 + (1 + 2C_g^2)\|\theta(t)\|_{\Theta}^2 + 2\|v(t)\|_{\Theta}^2\,.
            \end{equation}
            The terms involving $J(\theta)^*J(\theta)$ and $\varepsilon^2 I$ are dissipative in this estimate; therefore no separate growth assumption on $J(\theta)$ is needed for this particular global bound. 
            
            To write the right-hand side of \eqref{eq:GlobalExtProof:HelperA} in terms of $X(t)$ notice that both terms in \eqref{eq:ProofGlobalExsist:XDef} are non-negative so that
            \[
            \|\theta(t)\|_{\Theta}^2 \leq X(t)\,,\qquad \tau^2 \|v(t)\|_{\Theta}^2 \leq X(t)\,,
            \]
            and thus
            \begin{equation*}
             \dot{X}(t) \leq
            2C_g^2+\Bigl(1+2C_g^2+\frac{2}{\tau^2}\Bigr)X(t).
          \end{equation*}
          Gronwall's lemma therefore yields a bound on $X(t)$ on every finite time interval. Thus $(\theta(t),v(t))$ remains bounded on every finite time interval. Since the vector field is locally Lipschitz on the finite-dimensional phase space $\Theta\times\Theta$, the standard continuation theorem for ODEs implies that the maximal existence time is infinite \cite[Theorem~2.17]{TeschlODE}. 
        \end{proof}

    Proposition~\ref{prop:global_existence_dfo} shows that the nonuniqueness in the parameters given by (unregularized) Dirac--Frenkel dynamics \eqref{eq:DF_parameter_level} is avoided by DFI if an initial condition on $\theta$ and $v$ is prescribed.
        Instead of selecting, independently at each $\theta$, one element of the affine set $\mcS_0(\theta)$, DFI treats the parameter velocity as part of the state.
        Once the initial phase point $(\theta^0,v^0)$ is fixed, Proposition~\ref{prop:local_wellposedness_dfo} yields a unique parameter trajectory even for $\varepsilon=0$, and Proposition~\ref{prop:global_existence_dfo} shows that this trajectory exists for all times under the stated growth condition.

\subsection{A posteriori error analysis of DFI in continuous time}
\label{sec:continuous_error_analysis}

 Throughout this section we fix $\tau>0$ and $\varepsilon \geq 0$ and let $(\theta(t),v(t))$ be a sufficiently smooth solution of \eqref{eq:continuous_dfo}. For each $\theta$, recall the velocity $\bar{v}_{\varepsilon}(t) = \bar{v}_{\varepsilon}(\theta(t))$ given by \eqref{eq:TDFVelocity} for $\epsilon > 0$. For $\epsilon = 0$, consider any minimizer $\bar{v}_0(t)$. While $\bar{v}_0(t)$ need not be unique, $J(\theta(t))\bar{v}_0(t)$ is independent of the specific minimizer and thus we can define the corresponding 
\begin{equation}\label{eq:Posteriori:DotBarU}
\dot{\bar u}_\varepsilon(t):=J(\theta(t))\bar v_\varepsilon(t) 
\end{equation}
for $\epsilon \geq 0$. Note that $\bar{v}_0(t)$ might not depend continuously on time if $\varepsilon=0$; however, the following estimates depend on \eqref{eq:Posteriori:DotBarU} instead of $\bar{v}_0(t)$ directly. 

We consider the projection defect as 
        \begin{equation}
          \delta_\varepsilon(t)^2
          :=
          \norm{\dot{\bar u}_\varepsilon(t)-f(\theta(t))}_H^2
          +
          \varepsilon^2\norm{\bar v_\varepsilon(t)}_\Theta^2,
          \label{eq:delta_eps_cont}
        \end{equation}
        which is also used in \cite[Section 3.1]{Feischl2024}. Along the DFI trajectory, define the relaxation lag by
\[
r(t)
:=
J(\theta(t))\bigl(v(t)-\bar v_\varepsilon(t)\bigr)
=
\dot{\hat u}(t)-\dot{\bar u}_\varepsilon(t).
\]
  This motivates introducing the relaxation defect as
        \begin{equation}
          \rho_\varepsilon(t)^2
          :=
          \norm{r(t)}_H^2
          +
          \varepsilon^2\norm{v(t)-\bar v_\varepsilon(t)}_\Theta^2.
          \label{eq:rho_eps_cont}
        \end{equation}
   The following proposition shows that the instantaneous total defect can be decomposed into the projection defect $\delta_{\varepsilon}$ and the relaxation defect $\rho_{\varepsilon}$. 
   
   \begin{proposition}[Instantaneous defect decomposition]
          \label{prop:continuous_defect_decomposition}
          For every $t$ in the interval of existence,
          \begin{equation}
            \norm{\dot{\hat u}(t)-F(\hat u(t))}_H^2+\varepsilon^2\norm{v(t)}_\Theta^2
            =
            \delta_\varepsilon(t)^2+\rho_\varepsilon(t)^2.
            \label{eq:continuous_defect_identity}
          \end{equation}
        \end{proposition}
  \begin{proof}
          Decompose $v(t)$ as 
          \[
          v(t) = \bar{v}_{\varepsilon}(t) + (v(t) - \bar{v}_{\varepsilon}(t)) = \bar{v}_{\varepsilon}(t) + w(t)\,,
          \]
          to write 
          \begin{align*}
          \|\dot{\hat{u}}(t) - F(\hat{u}(t))\|_H^2 + \varepsilon^2 \|v(t)\|_{\Theta}^2 &= \|J(\theta(t))v(t) - f(\theta(t))\|_H^2 + \varepsilon^2\|v(t)\|_{\Theta}^2 \\
          & = \|J(\theta(t))\bar{v}_{\varepsilon}(t) - f(\theta(t)) + J(\theta(t))w(t)\|_H^2 + \varepsilon^2 \|\bar{v}_{\varepsilon}(t) + w(t)\|_{\Theta}^2\\
          & = (\|J(\theta(t))\bar{v}_{\varepsilon}(t) - f(\theta(t))\|_H^2 + \varepsilon^2\|\bar{v}_{\varepsilon}(t)\|^2_{\Theta})\\
          & + (\|J(\theta(t))w(t)\|_H^2 + \varepsilon^2\|w(t)\|_{\Theta}^2)\\
          & + 2\langle J(\theta(t))\bar{v}_{\varepsilon}(t) - f(\theta(t)), J(\theta(t))w(t)\rangle_H + 2\varepsilon^2\langle \bar{v}_{\varepsilon}(t), w(t)\rangle_{\Theta}\\
          & = \delta_{\varepsilon}(t)^2 + \rho_{\varepsilon}(t)^2\\
          & + 2\langle J(\theta(t))\bar{v}_{\varepsilon}(t) - f(\theta(t)), J(\theta(t))w(t)\rangle_H + 2\varepsilon^2\langle \bar{v}_{\varepsilon}(t), w(t)\rangle_{\Theta}\,. 
          \end{align*}
          We now show that the cross terms vanish, which then leads to the decomposition \eqref{eq:continuous_defect_identity}. Consider
          \[
          \langle J(\theta(t))\bar{v}_{\varepsilon}(t) - f(\theta(t)), J(\theta(t))w(t)\rangle_H = \langle J(\theta(t))^*(J(\theta(t))\bar{v}_{\varepsilon}(t) - f(\theta(t))), w(t)\rangle_{\Theta}
          \]
          and thus
          \begin{equation}\label{eq:InstantDefectDecomposition:HelperA}
          \begin{aligned}
          &\langle J(\theta(t))\bar{v}_{\varepsilon}(t) - f(\theta(t)), J(\theta(t))w(t)\rangle_H + \varepsilon^2\langle \bar{v}_{\varepsilon}(t), w(t)\rangle_{\Theta}\\
          &\qquad = \langle J(\theta(t))^*(J(\theta(t))\bar{v}_{\varepsilon}(t) - f(\theta(t))) + \varepsilon^2\bar{v}_{\varepsilon}(t), w(t)\rangle_{\Theta}\,.
          \end{aligned}
          \end{equation}
          The left argument of the inner product in \eqref{eq:InstantDefectDecomposition:HelperA} is the gradient that vanishes under the first-order optimality conditions of the objective \eqref{eq:Etheta}, which is minimized by 
         $\bar{v}_{\varepsilon}(t)$. Thus, the left argument of the inner product of \eqref{eq:InstantDefectDecomposition:HelperA} is zero and the cross terms vanish.
        \end{proof}

   \begin{theorem}[Continuous a posteriori error bound]
          \label{thm:continuous_aposteriori}We assume that the vector field $F$ satisfies the one-sided Lipschitz estimate
        \begin{equation}
          \ip{u-\widetilde u}{F(u)-F(\widetilde u)}_H
          \le
          \ell \norm{u-\widetilde u}_H^2,
          \qquad u,\widetilde u\in H,
          \label{eq:one_sided_lipschitz_cont}
        \end{equation}
        for some $\ell\in\R$.
          Let $u(t)$ be a sufficiently smooth solution of \eqref{eq:ode_intro} on $[0,T]$, and let $\hat u(t)=\Phi(\theta(t))$ be the DFI approximation. 
          Then, for every $t\in[0,T]$,
          \begin{equation}
            \norm{\hat u(t)-u(t)}_H
            \le
            e^{\ell t}\norm{\hat u(0)-u(0)}_H
            +
            \int_0^t e^{\ell(t-s)}
            \bigl(\delta_\varepsilon(s)^2+\rho_\varepsilon(s)^2\bigr)^{1/2}\,\dif s.
            \label{eq:continuous_aposteriori_bound}
          \end{equation}
          In particular, if $\hat u(0)=u(0)$, then
          \begin{equation}
            \norm{\hat u(t)-u(t)}_H
            \le
            \int_0^t e^{\ell(t-s)}
            \bigl(\delta_\varepsilon(s)^2+\rho_\varepsilon(s)^2\bigr)^{1/2}\,\dif s.
            \label{eq:continuous_aposteriori_exact_data}
          \end{equation}
        \end{theorem}
   \begin{proof}
          Define the error $e(t):=\hat u(t)-u(t)$.
          Because $u$ solves \eqref{eq:ode_intro},
          \begin{equation*}
            \dot e(t)=F(\hat u(t))-F(u(t))+\bigl(\dot{\hat u}(t)-F(\hat u(t))\bigr).
          \end{equation*}
        Now consider
          \begin{equation*}
            \frac12\frac{\dd}{\dd t}\norm{e(t)}_H^2 = \langle e(t), \dot{e}(t)\rangle_H
             = \langle \hat{u}(t) - u(t), F(\hat{u}(t)) - F(u(t)) \rangle_H + \langle e(t), \dot{\hat{u}}(t) - F(\hat{u}(t))\rangle_H\,.
          \end{equation*}
Using \eqref{eq:one_sided_lipschitz_cont}, we obtain 
\[
\frac12\frac{\dd}{\dd t}\norm{e(t)}_H^2 \leq \ell \|e(t)\|_H^2 + \langle e(t), \dot{\hat{u}}(t) - F(\hat{u}(t))\rangle_H\,.
\]
Applying Cauchy-Schwarz to the second term and using  \eqref{eq:continuous_defect_identity} to obtain  
          \begin{equation*}
            \norm{\dot{\hat u}(t)-F(\hat u(t))}_H^2
            \le
            \delta_\varepsilon(t)^2+\rho_\varepsilon(t)^2,
          \end{equation*}
        leads to
          \begin{equation*}
            \frac12\frac{\dd}{\dd t}\norm{e(t)}_H^2 
            \le
            \ell\norm{e(t)}_H^2
            +
            \norm{e(t)}_H
            \bigl(\delta_\varepsilon(t)^2+\rho_\varepsilon(t)^2\bigr)^{1/2}.
          \end{equation*}
          Whenever $\|e(t)\|_H\neq0$, division by $\norm{e(t)}_H$ yields
          \begin{equation*}
            \frac{\dd}{\dd t}\norm{e(t)}_H
            \le
            \ell\norm{e(t)}_H+\bigl(\delta_\varepsilon(t)^2+\rho_\varepsilon(t)^2\bigr)^{1/2},
          \end{equation*}
          and when $\|e(t)\|_H = 0$ then we restart the same argument from time $t$ when $\|e(t)\|_H \neq 0$.  
          By continuity this differential inequality extends to all $t\in[0,T]$, and Gronwall's lemma gives \eqref{eq:continuous_aposteriori_bound}.
          The second estimate is the specialization to exact initial data.
        \end{proof}

          Compared with the Tikhonov-regularized Dirac--Frenkel estimate \cite[Section~3.1]{Feischl2024}, this bound separates the instantaneous defect along the DFI trajectory into the regularized projection defect $\delta_\varepsilon$ and the relaxation defect $\rho_\varepsilon$;  however, it should not be read as the same estimate with an extra nonnegative term added along the same path.
          The quantities $\delta_\varepsilon$ and $\rho_\varepsilon$ are evaluated along the DFI trajectory $\theta(t)$, whereas the Tikhonov-regularized Dirac--Frenkel defect is evaluated along the trajectory generated by $\dot\theta=\bar v_\varepsilon(\theta)$.
          These trajectories may visit different regions of the parameter space $\Theta$, even when their function-space approximations are close.
          In particular, the inertial dynamics can move through kernel and near-kernel directions  and thereby in principle can sample representatives for which the instantaneous least-squares problem is better conditioned.
          Thus, $\delta_\varepsilon(t)$ may be smaller than the corresponding regularized Dirac--Frenkel defect along its own trajectory, but this is a trajectory-dependent effect rather than a pointwise comparison of the two bounds.

\section{Euler time discretization of DFI}
\label{sec:euler-dfi}
We now turn to the time discretization of DFI. We derive a time-discrete scheme based on a semi-implicit Euler discretization of \eqref{eq:continuous_dfo} that treats the velocity relaxation implicitly and the parameter update explicitly.
This choice yields a velocity update through an implicit solve, while keeping the parameter update explicit; equivalently, each step becomes a previous-velocity-anchored least-squares problem that exposes the inertial memory mechanism.

\subsection{Euler time discretization}
Let $h_k>0$ be the step size and $t_{k+1}=t_k+h_k$. Given $(\theta_k,v_k)$, we set
        \begin{equation}
          \hat u_k:=\Phi(\theta_k),
          \qquad J_k:=J(\theta_k),
          \qquad f_k:=f(\theta_k)=F(\hat u_k).
          \label{eq:frozen_quantities}
        \end{equation}
We discretize the DFI system \eqref{eq:continuous_dfo} with implicit Euler in $v$ and explicit Euler in $\theta$, 
        \begin{equation}\label{eq:DiscDFIEuler1}
        \begin{aligned}
          \theta_{k+1}= & \theta_k+h_k v_{k+1},\\
          \tau^2\frac{v_{k+1}-v_k}{h_k} = & -J_k^*(J_kv_{k+1}-f_k)-\varepsilon^2 v_{k+1}.
        \end{aligned}
        \end{equation}
It is convenient to further introduce
        \begin{equation}
\eta_k^2:=\varepsilon^2+\frac{\tau^2}{h_k} \in (0,\infty),
          \qquad
          \beta_k:=\frac{\tau^2}{\varepsilon^2h_k + \tau^2}\in(0,1],
          \qquad
          M_{\eta,k}:=J_k^*J_k+\eta_k^2 I.
          \label{eq:beta_eta_Mk_main}
        \end{equation} Note that $0<\beta_k<1$ when $\varepsilon>0$, while $\beta_k=1$ when $\varepsilon=0$.
        A direct computation shows that the discretized system \eqref{eq:DiscDFIEuler1} can be written as 
        \begin{equation}\label{eq:beta_eta_system_eta_beta}
        \begin{aligned}
          \theta_{k+1} & = \theta_k+h_k v_{k+1},                                    \\
          v_{k+1}      & = \beta_k v_k+M_{\eta,k}^{-1}J_k^*\bigl(f_k-\beta_k J_kv_k\bigr).
        \end{aligned}
        \end{equation}

The formulation \eqref{eq:beta_eta_system_eta_beta} mirrors the continuous direction-by-direction interpretation of DFI. 
If $q_{k,i}$ is a right singular vector of $J_k$ with singular value $\sigma_{k,i}$, and
\[
  v_{k,i}:=\ip{v_k}{q_{k,i}}_\Theta,
  \qquad
  g_{k,i}:=\ip{J_k^*f_k}{q_{k,i}}_\Theta,
\]
then
\[
  \ip{v_{k+1}}{q_{k,i}}_\Theta
  =
  \frac{g_{k,i}}{\sigma_{k,i}^2+\eta_k^2}
  +
  \frac{\beta_k\eta_k^2}{\sigma_{k,i}^2+\eta_k^2}v_{k,i}.
\]
Thus, for frozen $\theta_k$, the discrete update has the same type of direction-dependent balance as the continuous DFI dynamics in the sense that when $\sigma_{k,i}^2\gg\eta_k^2$, the memory factor $\beta_k\eta_k^2/(\sigma_{k,i}^2+\eta_k^2)$ 
is small, so the update is dominated by the current least-squares information. When $\sigma_{k,i}^2\ll\eta_k^2$, this factor is close to $\beta_k$, so the update retains most of the damped previous velocity coefficient.  
On $\ker J_k$, this reduces to $\ip{v_{k+1}}{q_{k,i}}_\Theta=\beta_k\ip{v_k}{q_{k,i}}_\Theta$, so nullspace motion is conserved when $\varepsilon=0$ and damped otherwise.

\subsection{Variational characterization of time-discrete DFI}
Let us now give a variational characterization of the DFI step \eqref{eq:beta_eta_system_eta_beta}, which will be useful for interpreting  the time-discrete DFI scheme and for the error analysis.

\begin{proposition}[Discrete variational characterization]
          \label{prop:discrete_variational_characterization}
          For fixed $k$, the velocity update $v_{k + 1}$ defined by the update \eqref{eq:beta_eta_system_eta_beta} is the unique minimizer of 
\begin{equation}
  \mcJ_k(w)
  :=
  \frac12\norm{J_kw-f_k}_H^2
  +
  \frac{\eta_k^2}{2}\norm{w-\beta_kv_k}_\Theta^2 .
  \label{eq:Jk_main_shifted}
\end{equation}
That is,  \begin{equation}
            v_{k+1}=\arg\min_{w\in \Theta}\mcJ_k(w) .
            \label{eq:main_minimization}
          \end{equation}
        \end{proposition}

\begin{proof}
         Since $\eta_k^2>0$, the functional $\mcJ_k$ is strongly convex and therefore has a unique minimizer. 
Its first-order optimality condition is
\[
  J_k^*(J_kw-f_k)+\eta_k^2(w-\beta_kv_k)=0 .
\]
For $w=v_{k+1}$ this becomes
\[
  (J_k^*J_k+\eta_k^2I)v_{k+1}
  =
  J_k^*f_k+\beta_k\eta_k^2v_k .
\]
Using $\beta_k\eta_k^2=\tau^2/h_k$, this becomes
\[
\left(J_k^*J_k+\varepsilon^2I+\frac{\tau^2}{h_k}I\right)v_{k+1}
=
J_k^*f_k+\frac{\tau^2}{h_k}v_k,
\]
which is precisely the velocity equation in \eqref{eq:DiscDFIEuler1}. 
Thus the unique minimizer of $\mcJ_k$ is the velocity given by the semi-implicit Euler step.
\end{proof}

The variational characterization shows that the DFI step is a (damped-)previous-velocity-anchored least-squares problem.  The new velocity is chosen with the aim to reduce the instantaneous Dirac--Frenkel residual while remaining close to the damped previous velocity $\beta_kv_k$.  Thus $\eta_k^2$ controls the strength of the anchoring, whereas $\beta_k$ controls how much of the previous velocity is retained in the anchor. We note that an analogous anchored least-squares problem is also obtained when using momentum in optimization with variational Monte Carlo  \cite{GOLDSHLAGER2024113351}.

\subsection{Algorithm}
With the change of variables $z=w-\beta_kv_k$ and after multiplying the objective by the factor 2, the minimization problem \eqref{eq:Jk_main_shifted} can be written as
\begin{equation}
  z_{k+1}
  =
  \arg\min_{z\in \Theta}
    \norm{J_k z -(f_k - \beta_k J_k v_k)}_H^2
    +
    \eta_k^2\norm{z}_\Theta^2
  ,
  \label{eq:main_minimization_shifted}
\end{equation}
and the velocity is recovered by
\[
  v_{k+1}=\beta_k v_k+z_{k+1}.
\]
The shifted form \eqref{eq:main_minimization_shifted} is convenient for implementation. It requires solving one standard Tikhonov-regularized least-squares problem for the correction $z_{k+1}$, while the retained part $\beta_kv_k$ carries the inertial memory of the method. The resulting semi-implicit Euler DFI algorithm is summarized in Algorithm~\ref{alg:DFI}.

\begin{algorithm}
  \caption{Semi-implicit Euler Dirac--Frenkel dynamics with inertia (DFI)}\label{alg:DFI}
  \begin{algorithmic}[1]
    \Require Initial data $\theta_0\in\Theta$, $v_0\in\Theta$, step sizes $h_k>0$, parameters $\tau>0$ and $\varepsilon\geq 0$
    \Ensure Iterates $(\hat u_k)_{k\geq 0}$
    \For{$k=0,1,2,\ldots$}
      \State Set $\hat u_k=\Phi(\theta_k)$, $J_k=J(\theta_k)$, and $f_k=F(\hat u_k)$.
      \State Set
      \[
        \eta_k^2=\varepsilon^2+\frac{\tau^2}{h_k},
        \qquad
        \beta_k=\frac{\tau^2}{\varepsilon^2h_k+\tau^2}.
      \]
      \State Solve the regularized least-squares problem
      \[
        z_{k+1}
        =
        \arg\min_{z\in\Theta}
          \norm{J_k z-(f_k-\beta_k J_kv_k)}_H^2
          +
          \eta_k^2\norm{z}_\Theta^2.
      \]
      \State Set $v_{k+1}=\beta_k v_k+z_{k+1}$.
      \State Set $\theta_{k+1}=\theta_k+h_kv_{k+1}$.
    \EndFor
  \end{algorithmic}
\end{algorithm}

\section{A posteriori error analysis of Euler-discretized DFI}
\label{sec:error_analysis}
We now provide an a posteriori analysis of the Euler-discretized DFI scheme \eqref{eq:beta_eta_system_eta_beta}. Throughout this section we fix $\tau>0$ and $\varepsilon>0$.
  \tc{The restriction $\varepsilon>0$ is used to control the norm of the velocity and to formulate a condition linking the step size, the local defect, and the regularization parameter. While we consider in this section only $\varepsilon > 0$, we note that similar arguments could be used for the case $\varepsilon=0$ provided that appropriate assumptions\footnote{For example, assume $c_{\Phi} = 0$ in Assumption~\ref{ass:flow_stability_new}(ii).} are made to keep the norm of the velocities bounded.}

\subsection{Local error bound}
We start by making stronger assumptions on the problem than in previous sections. We denote the flow of \eqref{eq:ode_intro} as
\begin{equation}
  \varphi_t:H\to H,
  \qquad u(t)=\varphi_t(u(0)).
  \label{eq:exact_flow}
\end{equation}
\begin{assumption}[Flow stability and regularity]
  \label{ass:flow_stability_new}
  Fix $T>0$.
  Assume that there exist constants $\ell\in\R$, $C_\Phi\ge0$, $c_\Phi\ge0$, and $C_a\ge0$ such that the following hold for all relevant states.
  \begin{enumerate}[label=\textnormal{(\roman*)},leftmargin=2.5em]
    \item The exact flow satisfies
          \begin{equation}
            \norm{\varphi_t(u)-\varphi_t(\widetilde u)}_H
            \le e^{\ell t}\norm{u-\widetilde u}_H,
            \qquad 0\le t\le T.
            \label{eq:flow_stability_new}
          \end{equation}
    \item The second derivative of $\Phi$ is bounded, relative to the frozen Jacobian $J(\theta)w:=D\Phi(\theta)[w]$ at the base point, as follows for all relevant $\theta,\zeta,\xi$:
          \begin{equation}
            \norm{D^2\Phi(\theta+s\zeta)[\xi,\xi]}_H
            \le
            C_\Phi\norm{J(\theta)\xi}_H^2
            +
            c_\Phi\norm{\xi}_\Theta^2,
            \qquad 0\le s\le 1.
            \label{eq:D2Phi_bound_new}
          \end{equation}
    \item For every relevant initial state $w$ and the corresponding exact trajectory $y(s)=\varphi_s(w)$, one has
          \[
            \norm{\ddot y(s)}_H\le C_a
          \]
          for all times $s$ for which the trajectory is used below.
  \end{enumerate}
\end{assumption}

We denote the discrete-time defect at time step $k$ as
\tc{
\begin{equation}
  \Delta_k^2
  :=
  \norm{J_kv_{k+1}-f_k}_H^2
  +
  \varepsilon^2\norm{v_{k+1}}_\Theta^2
  +
  \frac{\tau^2}{h_k}\norm{v_{k+1}-v_k}_\Theta^2\,.
  \label{eq:Delta_min_error}
\end{equation}
}
\tc{
The relation to the anchored least-squares objective \eqref{eq:Jk_main_shifted} is revealed by the identity
\begin{equation}
  \varepsilon^2\norm{v_{k+1}}_\Theta^2
  +
  \frac{\tau^2}{h_k}\norm{v_{k+1}-v_k}_\Theta^2
  =
  \eta_k^2\norm{v_{k+1}-\beta_kv_k}_\Theta^2
  +
  \varepsilon^2\beta_k\norm{v_k}_\Theta^2,
  \label{eq:Delta_augmented_identity_new}
\end{equation}
so that
\begin{equation}
  \Delta_k^2
  =
  2\mcJ_k(v_{k+1})
  +
  \varepsilon^2\beta_k\norm{v_k}_\Theta^2.
  \label{eq:Delta_objective_identity_new}
\end{equation}
}

\begin{lemma}[One-step local error]
  \label{lem:local_error_new}
  Under Assumption~\ref{ass:flow_stability_new}, for every time step $k$,
  the local error of one Euler step
  \begin{equation}\label{eq:one_step_uplus}
    \hat{u}_+ := \Phi(\theta_k + h_kv_{k + 1})\,.
  \end{equation}
  satisfies
  \begin{equation}
    \norm{\hat u_+-\varphi_{h_k}(\hat u_k)}_H
    \le
    h_k\,\Delta_k
    +
    C_\Phi h_k^2\norm{f_k}_H^2
    +
    \tc{
    h_k^2
    \left(
    C_\Phi+\frac{c_\Phi}{2\varepsilon^2}
    \right)\Delta_k^2}
    +
    \frac{C_a}{2}h_k^2 .
    \label{eq:local_error_bound_new}
  \end{equation}
\end{lemma}
\begin{proof}
  Let $ y_k(s):=\varphi_s(\hat u_k)$ for $0\le s\le h_k$.
  Then $y_k(0)=\hat u_k$ and, since $y_k$ solves the exact evolution equation, $\dot y_k(0)=F(\hat u_k)=f_k$.
  Taylor's formula with integral remainder gives
  \begin{equation}
    \varphi_{h_k}(\hat u_k)
    =
    y_k(h_k)
    =
    y_k(0)+h_k\dot y_k(0)
    +
    \int_0^{h_k}(h_k-s)\ddot y_k(s)\,ds .
    \label{eq:flow_taylor_integral_new}
  \end{equation}
  Hence
  \begin{equation}
    \varphi_{h_k}(\hat u_k)
    =
    \hat u_k+h_k f_k+r_k^{\mathrm{flow}},
    \label{eq:flow_taylor_new}
  \end{equation}
  where
  \begin{equation}
    r_k^{\mathrm{flow}}
    :=
    \int_0^{h_k}(h_k-s)\ddot y_k(s)\,ds .
    \label{eq:flow_remainder_def_new}
  \end{equation}
  By Assumption~\ref{ass:flow_stability_new}\textnormal{(iii)},
  \begin{align}
    \norm{r_k^{\mathrm{flow}}}_H
     & \le
    \int_0^{h_k}(h_k-s)\norm{\ddot y_k(s)}_H\,ds
    \notag \\
     & \le
    C_a\int_0^{h_k}(h_k-s)\,ds
    =
    \frac{C_a}{2}h_k^2 .
    \label{eq:flow_remainder_bound_new}
  \end{align}
  Next, Taylor's formula for $\Phi$ applied to the curve $s \mapsto \theta_k + sh_k v_{k + 1}$, which parametrizes the straight line segment in $\Theta$ from $\theta_k$ to $\theta_{k + 1} = \theta_k + h_k v_{k + 1}$,  gives
  \begin{align}
    \hat u_+
     & =
    \Phi(\theta_k+h_kv_{k+1})
    \notag \\
     & =
    \Phi(\theta_k)
    +
    h_kD\Phi(\theta_k)[v_{k+1}]
    +
    h_k^2
    \int_0^1
    (1-s)
    D^2\Phi(\theta_k+s h_kv_{k+1})
    [v_{k+1},v_{k+1}]
    \,ds .
    \label{eq:Phi_taylor_integral_new}
  \end{align}
  Since $\hat u_k=\Phi(\theta_k)$ and
  $J_kv_{k+1}=D\Phi(\theta_k)[v_{k+1}]$ (see \eqref{eq:Jf_def}), this becomes
  \begin{equation}
    \hat u_+
    =
    \hat u_k+h_kJ_kv_{k+1}+r_k^\Phi,
    \label{eq:Phi_taylor_new}
  \end{equation}
  where
  \begin{equation}
    r_k^\Phi
    :=
    h_k^2
    \int_0^1
    (1-s)
    D^2\Phi(\theta_k+s h_kv_{k+1})
    [v_{k+1},v_{k+1}]
    \,ds .
    \label{eq:Phi_remainder_def_new}
  \end{equation}
  Applying Assumption~\ref{ass:flow_stability_new}\textnormal{(ii)}
  with $
    \theta=\theta_k,
    \zeta=h_kv_{k+1},
    \xi=v_{k+1}$
  yields
  \begin{align}
    \norm{r_k^\Phi}_H
     & \le
    h_k^2
    \int_0^1
    (1-s)
    \norm{
    D^2\Phi(\theta_k+s h_kv_{k+1})
    [v_{k+1},v_{k+1}]
    }_H
    \,ds
    \notag \\
     & \le
    h_k^2
    \int_0^1
    (1-s)
    \left(
    C_\Phi\norm{J_kv_{k+1}}_H^2
    +
    c_\Phi\norm{v_{k+1}}_\Theta^2
    \right)
    \,ds
    \notag \\
     & =
    \frac{h_k^2}{2}
    \left(
    C_\Phi\norm{J_kv_{k+1}}_H^2
    +
    c_\Phi\norm{v_{k+1}}_\Theta^2
    \right).
    \label{eq:Phi_remainder_bound_new}
  \end{align}
  Let us now bound the state velocity $J_kv_{k+1}$.
    \tc{
      Since
      \begin{equation}
        J_kv_{k+1}=f_k+(J_kv_{k+1}-f_k),
      \end{equation}
      the definition of $\Delta_k$ gives
      \begin{equation} \label{eq:state_velocity_bound}
        \norm{J_kv_{k+1}}_H^2
        \le
        2\norm{f_k}_H^2
        +
        2\norm{J_kv_{k+1}-f_k}_H^2
        \le
        2\norm{f_k}_H^2
        +
        2\Delta_k^2.
      \end{equation}
      Moreover,
      \begin{equation} \label{eq:velocity_bound}
        \norm{v_{k+1}}_\Theta^2
        \le
        \frac{\Delta_k^2}{\varepsilon^2}.
      \end{equation}
    }

  We now subtract the exact-flow expansion
  \eqref{eq:flow_taylor_new} from the parametric expansion
  \eqref{eq:Phi_taylor_new} and obtain with the triangle inequality,
  \begin{align}
    \norm{\hat u_{k+1}-\varphi_{h_k}(\hat u_k)}_H
     & \le
    h_k\norm{J_kv_{k+1}-f_k}_H
    +
    \norm{r_k^\Phi}_H
    +
    \norm{r_k^{\mathrm{flow}}}_H .
    \label{eq:local_difference_bound_new}
  \end{align}
  Inserting the bounds
  \eqref{eq:flow_remainder_bound_new},
  \eqref{eq:Phi_remainder_bound_new},
  \tc{ \eqref{eq:state_velocity_bound},
      and \eqref{eq:velocity_bound}}, we obtain
  \begin{align}
    \norm{\hat u_{k+1}-\varphi_{h_k}(\hat u_k)}_H
     & \le
    h_k\norm{J_kv_{k+1}-f_k}_H
    +
    C_\Phi h_k^2\norm{f_k}_H^2
    +
    \tc{
    h_k^2
    \left(
    C_\Phi+\frac{c_\Phi}{2\varepsilon^2}
    \right)\Delta_k^2}
    +
    \frac{C_a}{2}h_k^2 .
    \label{eq:local_error_before_defect_new}
  \end{align}
  Finally, by using \eqref{eq:Delta_min_error}, we obtain the bound \eqref{eq:local_error_bound_new}.
\end{proof}

\tc{
The local-error bound in Lemma~\ref{lem:local_error_new} contains a quadratic defect contribution with a factor $\varepsilon^{-2}$, arising from the control of the norm of the velocity $v_{k+1}$.
To prevent this contribution from deteriorating as $\varepsilon$ becomes small, we impose the defect-control condition
}
\begin{equation}
  h_k\Delta_k\le c_\Delta\tc{\varepsilon^2}.
  \label{eq:DFIStepSizeControl}
\end{equation}
\tc{
Under this condition, the factor $\varepsilon^{-2}$ is compensated by $h_k\Delta_k$, and the corresponding contribution can be bounded by a constant multiple of $h_k\Delta_k$.
This is the same type of restriction as in the analysis of Tikhonov-regularized Dirac--Frenkel dynamics; see \cite{Feischl2024}.
Notably, its right-hand side contains only the Tikhonov contribution $\varepsilon^2$, rather than the effective regularization $\eta_k^2=\varepsilon^2+\tau^2/h_k$.
This is because the parameter-curvature term \eqref{eq:Phi_remainder_bound_new} involves the velocity $v_{k+1}$, which is controlled through $\varepsilon\norm{v_{k+1}}_\Theta\le\Delta_k$, whereas the inertial contribution controls only the increment $v_{k+1}-v_k$.
Thus, at the level of this stability argument, DFI and Tikhonov-regularized Dirac--Frenkel dynamics are subject to the same defect-control scale.
This does not preclude taking $\eta_k$ substantially larger than $\varepsilon$: when the DFI velocities vary coherently in time, the inertial contribution can increase the effective regularization without substantially increasing the defect.
We return to this point of view below.
}
\begin{proposition}[Local error under defect control]
  \label{prop:local_error_defect_control_new}
  Under the setup of Lemma~\ref{lem:local_error_new}, suppose that there exists a constant $c_\Delta\ge0$ such that \eqref{eq:DFIStepSizeControl} holds for all $k$.
  The local error of $\hat{u}_+$ given in \eqref{eq:one_step_uplus} satisfies
  \begin{align}
    \norm{\hat u_+-\varphi_{h_k}(\hat u_k)}_H
     & \le
    h_k\,\Delta_k
    +
    C_\Phi h_k^2\norm{f_k}_H^2
    +
    \tc{
    c_\Delta
    \left(
    C_\Phi\varepsilon^2+\frac{c_\Phi}{2}
    \right)h_k\Delta_k}
    +
    \frac{C_a}{2}h_k^2.
    \label{eq:defect_control_local_error_new}
  \end{align}
\end{proposition}
\begin{proof}
  \tc{
    Under \eqref{eq:DFIStepSizeControl},
    \begin{align}
      h_k^2
      \left(
      C_\Phi+\frac{c_\Phi}{2\varepsilon^2}
      \right)\Delta_k^2
       & =
      h_k\Delta_k
      \left(
      C_\Phi+\frac{c_\Phi}{2\varepsilon^2}
      \right)
      (h_k\Delta_k)
      \notag \\
       & \le
      c_\Delta
      \left(
      C_\Phi\varepsilon^2+\frac{c_\Phi}{2}
      \right)h_k\Delta_k.
    \end{align}
  }
  The result follows from \eqref{eq:local_error_bound_new}.
\end{proof}

\subsection{Global error}
The following global bound shows that the semi-implicit Euler DFI trajectory is controlled by the accumulated local DFI defects along the computed path. In the uniform-step case $h_k=h$, the estimate takes the form
\[
  \norm{\hat u_n-u(t_n)}_H
  \le
  C\bigl(\bar\Delta_{n-1}+h\bigr),
\]
where the constant $C$ depends on the flow stability constant in \eqref{eq:flow_stability_new}, the curvature constants of $\Phi$, the final time, and the DFI parameters. Thus, the non-defect contribution to the error is proportional to $h$. However, we strongly emphasize that this does not by itself prove that the error vanish under smaller time-step size $h \to 0$.

\begin{proposition}[A posteriori global error bound]
  \label{thm:global_error_new}
  Under the hypotheses of  Proposition~\ref{prop:local_error_defect_control_new}, let $u(t)$ denote the exact solution of \eqref{eq:ode_intro} with $u(0)=\hat u_0$.
  Then for every $k \geq 1$ with $t_k \leq T$,
  \begin{equation}
    \norm{\hat u_k-u(t_k)}_H
    \le
    \sum_{j=0}^{k-1}
    e^{\ell(t_k-t_{j+1})}
    \Biggl[
    h_j\Delta_j
    +
    C_\Phi h_j^2\norm{f_j}_H^2
    +
    \tc{
    c_\Delta
    \left(
    C_\Phi\varepsilon^2+\frac{c_\Phi}{2}
    \right)h_j\Delta_j}
    +
    \frac{C_a}{2}h_j^2
    \Biggr]\,,
    \label{eq:global_error_new}
  \end{equation}
  where $\ell$ is the constant from the flow stability in Assumption~\ref{ass:flow_stability_new}(i).
  \tc{
  Set
  \begin{equation*}
    \bar\Delta_{k-1}:=\max_{0\le j\le k-1}\Delta_j,
    \qquad
    \bar h:=\max_{0\le j\le k-1}h_j.
  \end{equation*}}
  Then, there exists a constant $C_T>0$, depending only on $T$, $\ell$, and universal numerical constants, such that for every $k$ with $t_k\le T$,
  \begin{equation}
    \norm{\hat u_k-u(t_k)}_H
    \le
    C_T
    \Biggl[
    \tc{
        \left(
        1
        +
        c_\Delta
        \left(
          C_\Phi\varepsilon^2+\frac{c_\Phi}{2}
          \right)
        \right)
        \bar\Delta_{k-1}}
    +
    \tc{
    \bar h
    \left(
    C_\Phi\bar F_{k-1}^2
    +
    \frac{C_a}{2}
    \right)}
    \Biggr]\,,
    \label{eq:global_bigO_new}
  \end{equation}
  where
  \begin{equation}
    \bar F_{k-1}:=\max_{0\le j\le k-1}\norm{f_j}_H.
    \label{eq:bar_F_def_new}
  \end{equation}
\end{proposition}
\begin{proof}
  Fix $k\ge1$ with $t_k\le T$. The index $k$ denotes the final time at which we want to estimate the global error $\|\hat{u}_k - u(t_k)\|_H$. We use $j=0,\dots,k-1$ to index the individual time steps whose local errors are propagated to the final time $t_k$.

  Since $u(0) = \hat{u}_0$, we have $u(t_k) = \varphi_{t_k}(\hat{u}_0)$. Now we introduce the intermediate propagated states
  \[
    W_j := \varphi_{t_k - t_j}(\hat{u}_j)\,,\qquad j = 0, \dots, k\,.
  \]
  At the end points, we have $W_0 = \varphi_{t_k}(\hat{u}_0) = u(t_k)$ and $W_k = \varphi_0(\hat{u}_k) = \hat{u}_k$. Therefore $\hat{u}_k - u(t_k) = W_k - W_0$ and telescoping (Lady Windermere's fan) gives us
  \begin{equation}\label{eq:Proof:TelescopingglobalError}
    \hat{u}_k - u(t_k) = W_k - W_0 = \sum_{j = 0}^{k - 1}W_{j + 1} - W_j\,.
  \end{equation}
  Now use the semigroup property $\varphi_{t_k - t_j} = \varphi_{t_k - t_{j + 1}} \circ \varphi_{h_j}$ of the exact flow $\varphi$ to obtain
  \[
    W_{j+1}-W_j
    =
    \varphi_{t_k-t_{j+1}}(\hat u_{j+1})
    -
    \varphi_{t_k-t_{j+1}}(\varphi_{h_j}(\hat u_j)).
  \]
  Therefore, by the triangle inequality and flow stability,
  \begin{equation}\label{eq:Proof:TelescopingWithTriangle}
    \norm{\hat u_k-u(t_k)}_H
    \le
    \sum_{j=0}^{k-1}
    e^{\ell(t_k-t_{j+1})}
    \norm{\hat u_{j+1}-\varphi_{h_j}(\hat u_j)}_H,
  \end{equation}
  where we took the $H$-norm,  applied the triangle inequality, and used the flow stability  \eqref{eq:flow_stability_new}. Now notice that the terms $\|\hat{u}_{j + 1} - \varphi_{h_j}(\hat{u}_j)\|_H$ denote local errors in \eqref{eq:Proof:TelescopingWithTriangle}, so we can substitute the bound \eqref{eq:defect_control_local_error_new} from Proposition~\ref{prop:local_error_defect_control_new} into \eqref{eq:Proof:TelescopingWithTriangle} to obtain \eqref{eq:global_error_new}.

  It remains to prove the simplified estimate
  \eqref{eq:global_bigO_new}. Let $\ell_+:=\max\{\ell,0\}$.
  Since $0\le t_k-t_{j+1}\le T$, we have
  \begin{equation}
    \sum_{j=0}^{k-1}
    e^{\ell(t_k-t_{j+1})}\tc{h_j}
    \le
    e^{\ell_+T}
    \sum_{j=0}^{k-1}\tc{h_j}
    =
    e^{\ell_+T}t_k
    \le
    e^{\ell_+T}T.
    \label{eq:global_weighted_h_sum_new}
  \end{equation}
  For $0\le j\le k-1$, we have
  \[
    \Delta_j\le\bar\Delta_{k-1},
    \qquad
    \norm{f_j}_H\le\bar F_{k-1},
    \qquad
    \tc{h_j^2\le\bar h h_j}.
  \]
  Applying these estimates to
  \eqref{eq:global_error_new}, and using
  \eqref{eq:global_weighted_h_sum_new}, gives \eqref{eq:global_bigO_new}
  \tc{with $C_T=e^{\ell_+T}T$}.
\end{proof}

\subsection{Interpretation and discussion of the bounds}
\tc{
  We now discuss that, under a suitable $h$ scaling regime of the involved parameters and quantities, DFI can use a higher effective regularization, given by the parameter $\eta_k$, than the Tikhonov regularization, controlled by the parameter $\varepsilon$, while retaining an $\mathcal{O}(h)$ defect.
  
  To ease exposition, let us consider a uniform step size $h$, hence a uniform $\eta$. 
  Suppose that the following scaling holds along the DFI trajectory,
  \begin{equation}
    \norm{J_kv_{k+1}-f_k}_H=\mathcal{O}(h),
    \qquad
    \norm{v_{k+1}}_\Theta=\mathcal{O}(1),
    \qquad
    \norm{v_{k+1}-v_k}_\Theta=\mathcal{O}(h^\rho),
  \end{equation}
  for $\rho > 0$. 
  The first condition expresses that the residual norm is of order $h$, that is the DFI velocity gives a first order approximation of the function-space velocity $f_k$. 
  The third condition, instead, quantifies the temporal coherence of the DFI velocities in the sense that a larger exponent $\rho$ means that successive velocities approach one another more rapidly as $h\to0$, which corresponds to greater coherence. 
  Furthermore, let the parameters scale with $h$ as
  \begin{equation}
    \varepsilon\asymp h,
    \qquad
    \tau\asymp h^\alpha,
  \end{equation}
  where the notation $f(h) \asymp h$ means that there exist constants $C \geq c > 0$ such that $ch \leq f(h) \leq Ch$ for all sufficiently small values of $h$.
  
  Let us first consider how the defect scales with $h$ given the scalings discussed above. A direct calculation gives
  \begin{align*}
    \Delta_k^2
     & =
    \norm{J_kv_{k+1}-f_k}_H^2
    +
    \varepsilon^2\norm{v_{k+1}}_\Theta^2
    +
    \frac{\tau^2}{h}\norm{v_{k+1}-v_k}_\Theta^2
     \in 
    \mathcal{O}(h^2)+\mathcal{O}\left(h^{2(\alpha+\rho)-1}\right).
  \end{align*}
  Thus, the condition
  \begin{equation} \label{eq:cond_alpha_rho}
    \alpha+\rho\geq\frac{3}{2}
  \end{equation}
  implies $\Delta_k\in \mathcal{O}(h)$.
  Note also that, since $h\Delta_k\in\mathcal{O}(h^2)$ and $\varepsilon^2\asymp h^2$ in this regime, the defect-control condition \eqref{eq:DFIStepSizeControl} is satisfied for an appropriate constant. 
  Hence, Proposition~\ref{thm:global_error_new} gives an $\mathcal{O}(h)$ state error.

  Importantly, while the defect is of order $h$, under the scaling above, the effective regularization $\eta$ can be larger. 
  Indeed, we have 
  \begin{equation*}
    \eta^2
    =
    \varepsilon^2+\frac{\tau^2}{h}
    \asymp
    h^2+h^{2\alpha-1},
    \quad
    1-\beta = \frac{\varepsilon^2}{\eta^2} \asymp h^{3 - 2\alpha}
    .
  \end{equation*}
  Since $\rho > 0$, condition \eqref{eq:cond_alpha_rho} allows us to choose $\alpha < 3/2$ so that 
  $\Delta_k=\mathcal{O}(h)$ while $\eta$ is asymptotically larger than $\varepsilon$.
  For example, first-order coherence $\rho=1$ together with $\tau\asymp h$ gives
  \begin{equation*}
    \varepsilon\asymp h,
    \qquad
    \eta\asymp h^{1/2},
    \qquad
    \Delta_k \in \mathcal{O}(h).
  \end{equation*}

  The effect of a larger $\eta$ can be exploited in various algorithmic ways to make the computation of an accurate approximation of the DFI velocity cheaper.
  Indeed, in Algorithm~\ref{alg:DFI},  the parameter update involves the solution of a regularized least-squares problem with associated normal equations operator $M_{\eta,k}$ introduced in \eqref{eq:beta_eta_Mk_main}.
  If $\sigma_{k,i}$ are the singular values of $J_k$, then the condition number $\kappa_k(\eta)$ and effective dimension $d_{\mathrm{eff},k}(\eta)$ of $M_{\eta,k}$, defined as
  \begin{equation}
    \kappa_k(\eta)
    =
    \frac{\sigma_{k,1}^2+\eta^2}
    {\sigma_{k,p}^2+\eta^2},
    \qquad
    d_{\mathrm{eff},k}(\eta)
    =
    \sum_i
    \frac{\sigma_{k,i}^2}
    {\sigma_{k,i}^2+\eta^2},
  \end{equation}
  both decrease as $\eta$ increases.
  Better conditioning may reduce the number of iterations required by an iterative solver. A smaller effective dimension may be exploited through reduced-rank approximations, truncation, randomized linear algebra, or other solvers adapted to the numerical spectrum. Consider randomized sketching.
  For suitable embeddings, the sketch dimension required to approximate a regularized least-squares problem can scale, up to accuracy and logarithmic factors, with its effective dimension \cite{lacotte2020effective,pmlr-v80-lin18b}.
  Since increasing $\eta$ decreases the effective dimension, it can permit a smaller and hence cheaper-to-solve sketched problem; see our numerical experiments in Section~\ref{sec:numerics}.

  In stark contrast, in Tikhonov-regularized Dirac--Frenkel dynamics, the parameter update is computed solving a least-squares problem with regularization parameter $\varepsilon$. Increasing $\varepsilon$ above the scale $\mathcal{O}(h)$ increases the value of the Tikhonov-regularized objective above the scale $\mathcal{O}(h)$, hence failing to achieve $\mathcal{O}(h)$ accuracy of the integration scheme.
}

\section{Numerical experiments}
\label{sec:numerics}
We demonstrate the DFI scheme on examples with the Allen--Cahn and Fokker-Planck equations.

\subsection{Allen--Cahn equation}
\label{sec:numexp_allen_cahn}

\subsubsection{Setup}
We consider the one-dimensional Allen--Cahn equation on the periodic domain $[0,2\pi)$,
\begin{equation}\label{eq:NumExp:AllenCahn}
    \partial_t u(t,x)
    =
    \nu \partial_{xx} u(t,x)+u(t,x)-u(t,x)^3,
    \qquad
    (t,x)\in[0,T]\times[0,2\pi),
\end{equation}
with periodic boundary conditions. Recall that \eqref{eq:NumExp:AllenCahn} is an $L^2$-gradient flow because if we set
\begin{equation*}
    V(u)
    =
    \int_0^{2\pi}
    \frac{\nu}{2}\abs{\partial_x u(x)}^2
    +
    \frac{1}{4}\left(u(x)^2-1\right)^2
    \dif x,
\end{equation*}
then $\partial_t u = -\nabla_u V(u)$. We set the viscosity parameter to $\nu=0.2$ and consider the end time $T=15$. The initial condition is
\begin{equation*}
    u^0(x)
    =
    \frac{1}{3}\tanh(2\sin x)
    -e^{-23.5\left(x-\frac{\pi}{2}\right)^2}
    +e^{-27(x-4.2)^2}
    +e^{-38(x-5.4)^2},
\end{equation*}
which is numerically periodic up to double precision.

\subsubsection{Nonlinear parametrization with neural network}\label{sec:NumExp:AC:DNNSetup}
We parametrize the approximate solution by $\hat u_k(x)=\Phi(\theta_k)(x)$,  where $\Phi(\theta)$ is a periodic fully connected feedforward neural network.  The periodicity is imposed at the input level. Namely, instead of feeding $x$ directly into the network, we first define the periodic feature vector
\[
    \phi(x)
    =
    \bigl(\phi_1(x),\ldots,\phi_{16}(x)\bigr)\in\mathbb R^{16},
    \qquad
    \phi_j(x)
    =
    \cos\left(\frac{2\pi x}{L}+s_j\right),
    \qquad
    L=2\pi,
\]
where the shifts $s_j$ are fixed. Since $L=2\pi$, these features satisfy $\phi_j(x+L)=\phi_j(x)$, and therefore every function obtained by composing a standard feedforward network with $\phi(x)$ is $L$-periodic. The network has five hidden layers, each of width $15$, and uses the swish activation applied componentwise. The trainable parameters are the entries of the weight matrices, biases, output weights, and output bias. Their total number is $p = 1231$. Hence the parameter space is $\Theta=\mathbb R^{1231}$.

\subsubsection{Setup of DFI scheme}
We use the $L^2(0,2\pi)$ inner product in space. In the numerical experiments this inner product is approximated on the uniform periodic grid $x_0, \dots, x_{N_x - 1}$ with $N_x = 450$ points, using
\begin{equation*}
    \norm{w}_{2,N_x}^2
    =
    \frac{1}{N_x}
    \sum_{i=0}^{N_x-1}\abs{w(x_i)}^2.
\end{equation*}
In the empirical least-squares solves, both $J_k$ and $f_k$ are evaluated on this grid using JAX's automatic differentiation to compute all derivatives involved. The semi-implicit Euler DFI update is executed as described in Algorithm~\ref{alg:DFI}.
The initialization parameter $\theta_0$ is computed by fitting the initial condition with the Adam optimizer run for $5 \times 10^4$ full batch iterations on the same grid, while the initial velocity is set to $v_0=0$. Furthermore, we use a fixed time step $h$ so that
\[
    t_k=kh,
    \qquad
    k=0,\ldots,N_t,
    \qquad
    N_t=\frac{T}{h}.
\]

We also consider the left-sketched version of DFI and Tikhonov-regularized Dirac--Frenkel dynamics, where the correction is obtained by solving
\begin{align}
    z_{k+1}
     & =
    \arg\min_{z\in \R^p}
    \norm{
        S_k \sqrt{W} \left(J_kz-(f_k-\beta J_kv_k)\right)
    }_2^2
    +
    \eta^2\norm{z}_{\R^p}^2,
    \label{eq:dfi_spat_disc_left_sketch}
    \\
    v_{k+1}
     & =
    \beta v_k+z_{k+1}\,,
\end{align}
where $W=\mathrm{diag}(1/N_x,\dots,1/N_x)$ denotes the quadrature weights matrix, and $S_k\in\R^{s\times N_x}$ is a left-sketch matrix, $s \leq N_x$, such that $\mathbb{E}[\|Sw\|_2^2]= \|w\|_2^2$.

The reference solution, denoted by $u_k$, is computed to near machine precision with a Fourier spectral discretization in space and fourth-order Runge--Kutta in time with step size $h_{\mathrm{ref}}=2.5\cdot 10^{-4}$. We regard this trajectory as the true solution for reporting error diagnostics. We report the pointwise relative error given by
\begin{equation}
    e_k
    =
    \frac{
        \norm{\hat u_k-u_k}_{2,N_x}
    }{
        \norm{u_k}_{2,N_x}
    }.\label{eq:NumExp:PointwiseError}
\end{equation}
and the integrated relative error,
\begin{equation*}
    E
    =
    \left(
    \frac{
        \sum_{k=0}^{N_t-1}h\,
        \norm{\hat u_k-u_k}_{2,N_x}^2
    }{
        \sum_{k=0}^{N_t-1}h\,
        \norm{u_k}_{2,N_x}^2
    }
    \right)^{1/2}.
\end{equation*}

\begin{figure}
    \centering
    \begin{tabular}{cc}
        \includegraphics[width=0.55\textwidth]{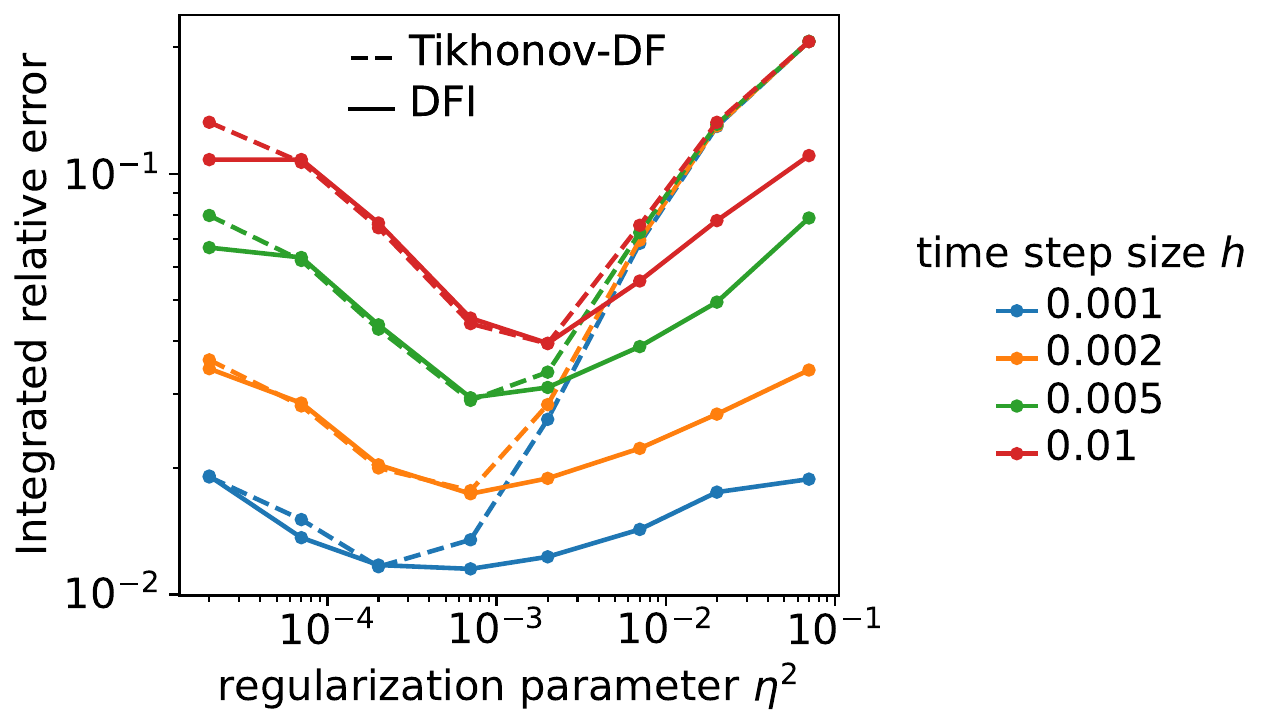} &
        \includegraphics[width=0.35\textwidth]{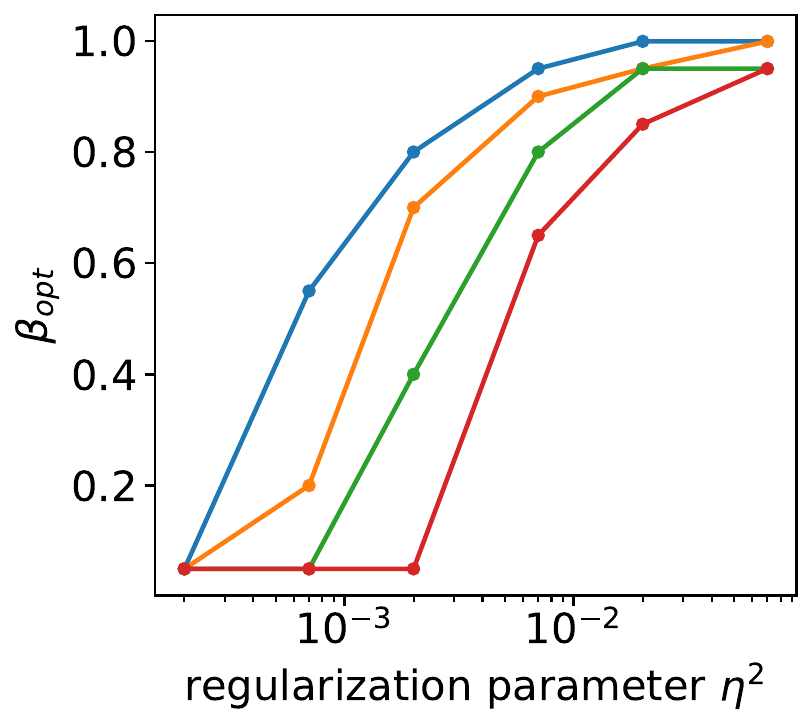}                                                                \\
        \scriptsize (a) integrated relative error                                      & \scriptsize (b) optimal memory parameter $\beta_{\text{opt}}$
    \end{tabular}
    \caption{DFI uses memory to compensate for loss of information in the instantaneous regularized least-squares solve. Plot (a) shows that as $\eta^2$ increases, Tikhonov-DF deteriorates because the velocity is computed from an increasingly regularized local problem alone, while DFI remains accurate by retaining more of the previous velocity. Plot (b) provides more evidence of this mechanism, showing that the error-minimizing $\beta_{\text{opt}}$ increases with $\eta^2$, meaning that more previous-velocity information is used when the instantaneous least-squares signal is more strongly regularized} \label{fig:allen_cahn_interr_and_beta_opt_vs_eta2}
\end{figure}

\begin{figure}
    \centering
    \begin{tabular}{cc}
        \includegraphics[width=0.55\textwidth]{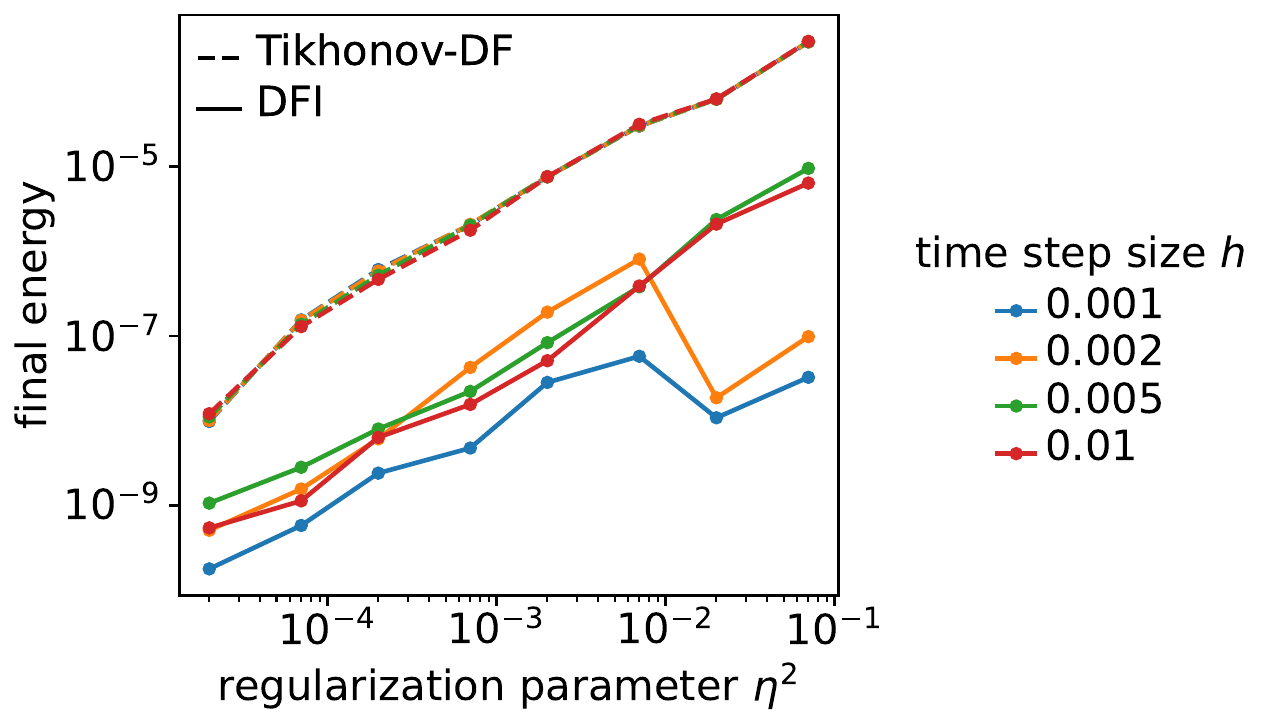} &
        \includegraphics[width=0.38\textwidth]{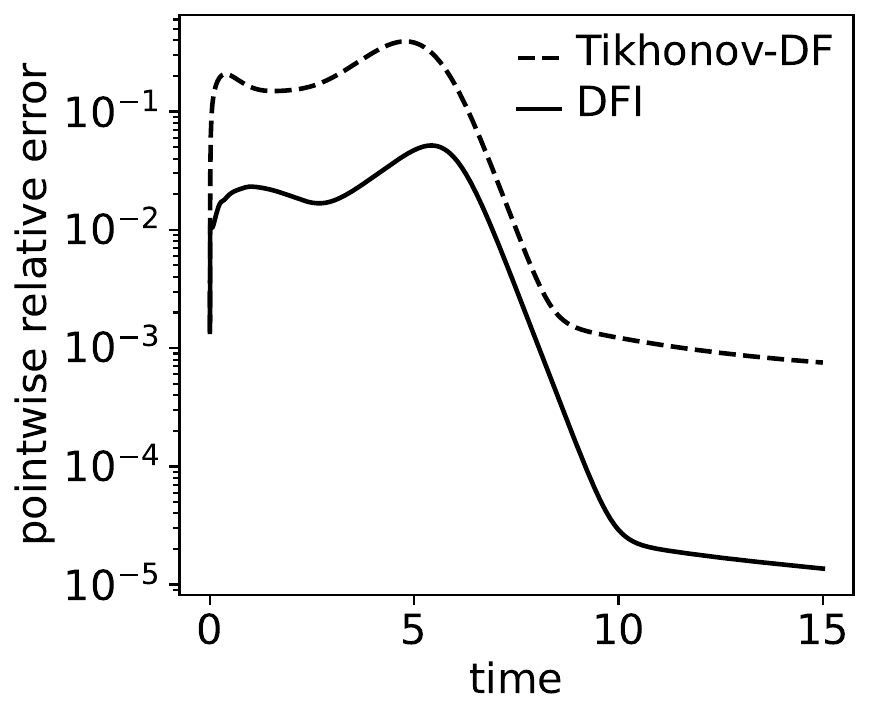}                                                                                   \\
        \scriptsize (a) final energy                                                                         & \scriptsize (b) pointwise error for $\eta^2 = 2 \times 10^{-2}$
    \end{tabular}
    \caption{Plot (a) shows that DFI reaches a smaller final energy than Tikhonov-DF across regularization strengths, indicating that the inertial dynamics better follow the long-time energy decay in this example. Plot (b) shows that this improvement is not only a final-time effect. After the initial transient, DFI also gives smaller pointwise errors along the trajectory.}
    \label{fig:allen_cahn_final_energy_gap_vs_eta2}
\end{figure}

\subsubsection{Regularization parameters}\label{sec:NumExp:AllenCahn:Parameters}
For the numerical experiments, we directly select $\beta$ and $\eta$ because they are interpretable as the memory and effective regularization parameter. We recall that, for a time-step size $h$, these parameters correspond to the  parameters used in Algorithm~\ref{alg:DFI} through $\tau^2=\beta\eta^2h$ and $\epsilon^2=(1-\beta)\eta^2$.
We compare DFI to Tikhonov-regularized Dirac--Frenkel dynamics, which corresponds to the same discrete-time update with formally setting $\beta=0$ in Algorithm~\ref{alg:DFI}. Thus, for both schemes, we need to select $\eta$ but for DFI we additionally need to select $\beta$. While this is an additional parameter to select for DFI, we note that $\beta$ has to be in the range $(0, 1)$ and that it is directly interpretable as controlling how much memory is taken into account. We will explore the effect of $\beta$ in the following experiments.

\subsubsection{Benefit of inertia: robustness when the local Jacobian information is limited} \label{sec:NumExp:JacobianInertiaAllenCahn}
The DFI scheme remains accurate even when the instantaneous least-squares problem built from the instantaneous Jacobian $J_k$ provides only limited reliable information about the next velocity. As we discussed earlier, this can happen because $J_k$ is ill-conditioned or nearly rank-deficient, so some parameter directions are only weakly visible in the tangent vector $J_kv$. Additionally, the least-squares problem may be too strongly  regularized locally, which suppresses the correction obtained from the instantaneous Jacobian. We also note a third case, which is that the residual may be evaluated only through a left sketch $S_kJ_k$ as in \eqref{eq:dfi_spat_disc_left_sketch}, so the update uses a compressed approximation of the least-squares problem, which can reduce computational costs  \cite{berman2023randomized,dong2025randomizedtimesteppingnonlinearly}. In all three cases, the instantaneous local solve contains less usable information about the velocity that should be taken in parameter space. In Tikhonov-DF, weakly informed velocity directions are shrunk by the regularized instantaneous solve. By contrast, DFI combines the information available from the instantaneous Dirac--Frenkel residual with the history given by the inertia of the parameter velocity.

This mechanism explains the behavior we see in  Figure~\ref{fig:allen_cahn_interr_and_beta_opt_vs_eta2}. The left plot shows that DFI is more robust than Tikhonov-DF as the regularization strength $\eta^2$ increases. For small and moderate values of $\eta^2$, the two methods give comparable integrated errors. When $\eta^2$ becomes large, the correction obtained from the instantaneous least-squares problem is strongly penalized. In this regime, the Tikhonov-DF error grows rapidly because Tikhonov-DF recomputes a shrunk velocity from scratch at each step. DFI, on the other hand, remains accurate over a wider range of $\eta^2$ because the transported velocity $\beta v_k$ supplies information that is not obtained from the instantaneous regularized correction alone. The right plot in Figure~\ref{fig:allen_cahn_interr_and_beta_opt_vs_eta2} provides further evidence that this is consistent with this interpretation: For this we note that the parameter $\beta$ has been selected via a grid search over the values $\{0.05,0.1,\dots,0.95,0.99\}$ to minimize the error on a reference solution for a given $\eta^2$. The plot shows that the optimal value of $\beta$ to minimize the integrated error $E$ increases with $\eta^2$. Thus, when the correction from the instantaneous Jacobian is more strongly regularized, the best DFI trajectory compensates by retaining more of the previous velocity. This matches the direction-wise interpretation from \Cref{sec:DFI:Interpretation}. Directions that are well resolved by $J_k$ are updated using the instantaneous Dirac--Frenkel residual, while weakly resolved or strongly damped directions can be carried forward by inertia instead of being instantaneously suppressed.

The gradient-flow diagnostics in Figure~\ref{fig:allen_cahn_final_energy_gap_vs_eta2} show the same behavior from the perspective of the energy. Across different regularization strengths, DFI reaches a smaller final energy than Tikhonov-DF, suggesting that the inertial dynamics help preserve the long-time relaxation structure of the Allen--Cahn flow.
The pointwise-in-time comparison in Figure~\ref{fig:allen_cahn_final_energy_gap_vs_eta2} shows that the improvement is not only a final-time effect but DFI also reduces the pointwise relative error after the initial transient.

\begin{SCfigure}
    \centering
    \includegraphics[width=0.58\textwidth]{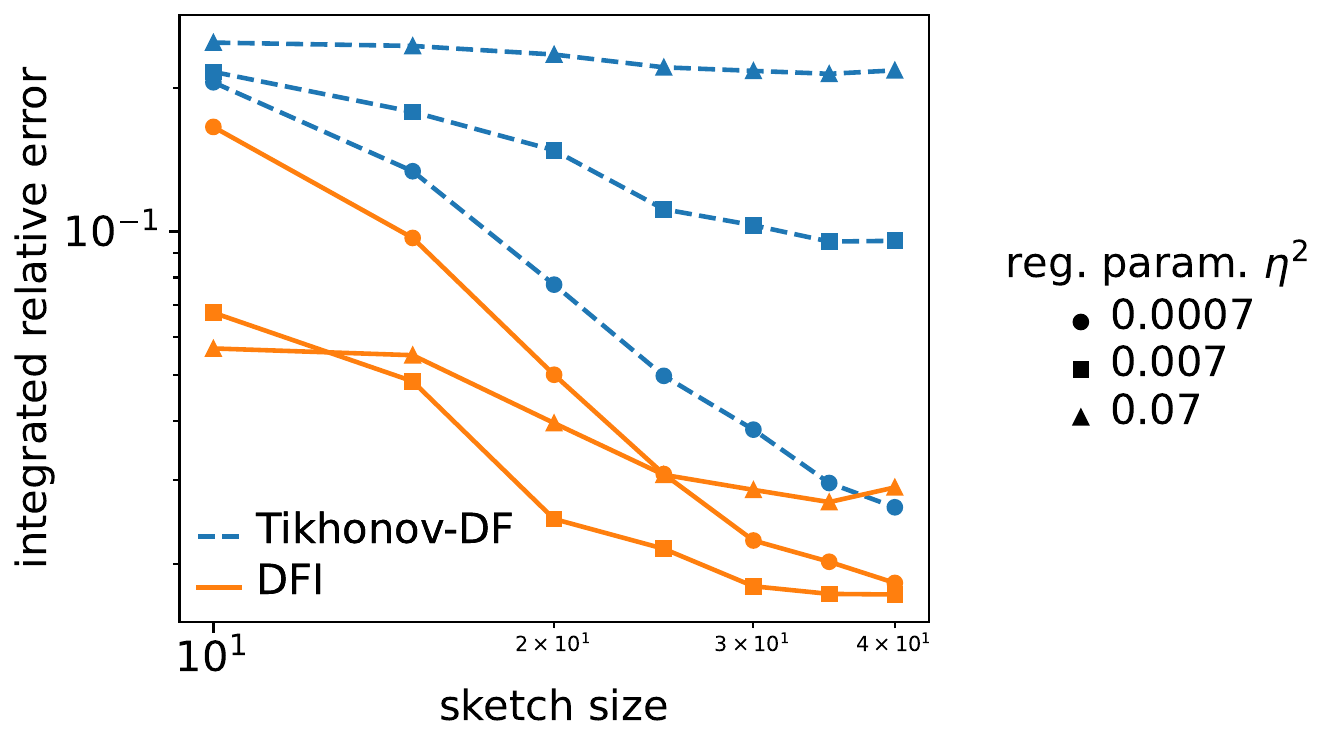}
    \caption{As the sketch size $s$ decreases, the local solve uses less information from the residual, so Tikhonov-DF deteriorates because it relies entirely on the instantaneous sketched least-squares problem. DFI remains more accurate for small sketch sizes because the velocity history supplements the information missing from the instantaneous sketched solve.}
    \label{fig:allen_cahn_interr_vs_sketch_size}
\end{SCfigure}

\subsubsection{DFI is more robust under sketching}
A smaller sketch size $s$ reduces the cost of the least-squares solve, but it also means that the method uses less information from the residual. Since Tikhonov-DF relies entirely on this instantaneous sketched solve, its accuracy is more sensitive to a small sketch size. DFI is less sensitive because it also uses the velocity from the past. Thus, DFI can maintain accuracy even when the local least-squares problem is made less informative in order to reduce cost. The results in Figure~\ref{fig:allen_cahn_interr_vs_sketch_size} demonstrate this effect. For small sketch sizes, DFI achieves lower integrated errors than Tikhonov-DF. Thus the inertial memory is not merely a qualitative difference in the dynamics; it leads to a computational benefit, allowing one to use cheaper sketched Dirac--Frenkel updates with less loss of accuracy. Note that the parameter $\beta$ has been selected as in  \Cref{sec:NumExp:JacobianInertiaAllenCahn}.

\subsubsection{Inertia helps, but excessive memory hurts}
The results in Figure~\ref{fig:allen_cahn_interr_vs_beta} show the expected tradeoff in the memory parameter $\beta$ for a fixed $\eta^2 = 2 \times 10^{-4}$. Increasing $\beta$ allows DFI to retain more information from the previous velocity, which is the main mechanism behind its robustness. However, taking $\beta$ too close to one makes the method insufficiently responsive to the instantaneous Dirac--Frenkel residual. Then inaccurate or outdated velocity components can persist for too long, and the integrated error grows. \tc{This behavior is consistent with the defect-based perspective of \Cref{sec:error_analysis}. If the anchor velocity becomes outdated, the residual  and hence the defect can grow. Moreover, at fixed effective regularization parameter $\eta$, increasing $\beta$ decreases $\varepsilon^2=(1-\beta)\eta^2$ and therefore makes the defect-control condition \eqref{eq:DFIStepSizeControl} increasingly restrictive.}

\begin{SCfigure}[1.0][t]
    \centering
    \includegraphics[width=0.48\textwidth]{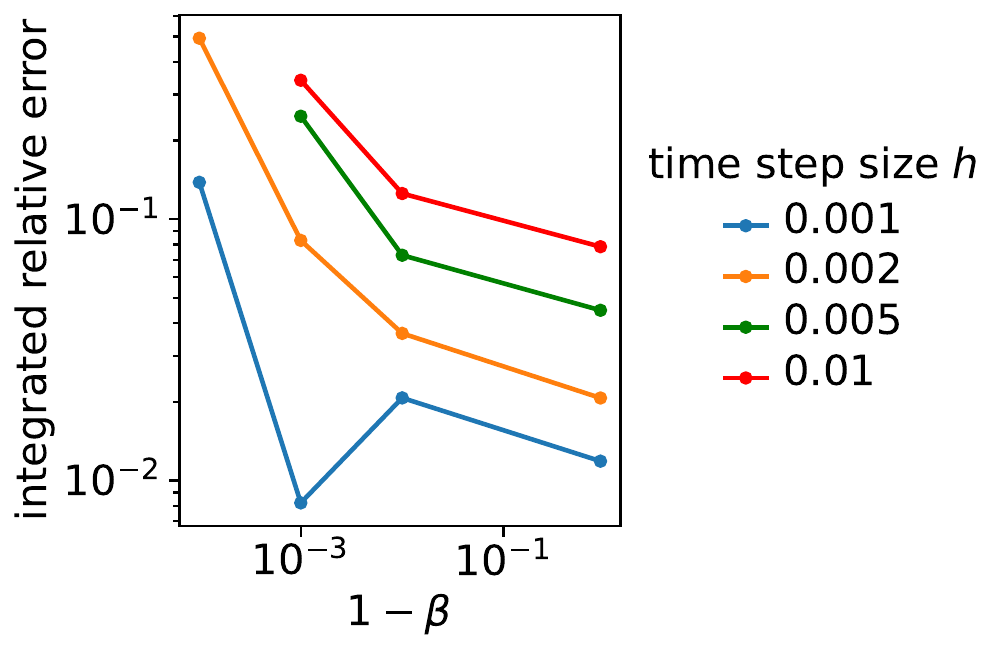}
    \caption{As $\beta$ approaches one, equivalently as $1-\beta$ becomes small, the transported velocity dominates and the instantaneous least-squares correction has too little influence. \tc{In this regime, outdated velocity information can increase the instantaneous residual and hence the augmented defect, while the decrease of $\varepsilon^2=(1-\beta)\eta^2$ makes the defect-control condition more restrictive.}}
    \label{fig:allen_cahn_interr_vs_beta}
\end{SCfigure}

\subsection{High-dimensional Fokker--Planck equation}
\label{sec:numexp_fokker_planck}

\subsubsection{Setup}
We consider a Fokker--Planck equation for a probability density
$u(t,\cdot)$ on $\R^d$, with $d=10$,
\begin{equation}
    \begin{aligned}
        \partial_t u(t,x) =
        -\nabla_x\cdot\left(u(t,x)b(t,x)\right)
        +
        D\Delta_x u(t,x),
        \qquad
        x\in\R^d,
        \quad
        t\in[0,T].
    \end{aligned}
    \label{eq:NumExp:FokkerPlanck}
\end{equation}
Here $b(t,x)\in\R^d$ is the drift field and $D>0$ is the diffusion coefficient. The equation describes the evolution of the probability density of the
stochastic differential equation
\begin{equation}
    \dif X_i(t)
    =
    b_i(t,X(t))\,\dif t
    +
    \sqrt{2D}\,\dif W_i(t),
    \qquad
    i=1,\ldots,d.
    \label{eq:fpe_sde}
\end{equation}
We set $D=10^{-2}$ and integrate until $T=2$.

The drift describes an interacting anharmonic trap \cite{bruna_neural_2024}. Its
components are
\begin{equation*}
    b_i(t,x)
    =
    (a(t)-x_i)^3+\alpha(\bar x-x_i),
    \qquad
    \bar x=\frac1d\sum_{j=1}^d x_j,
    \qquad
    i=1,\ldots,d,
\end{equation*}
where
\begin{equation*}
    a(t)=1.25(\sin(\pi t)+1.5),
    \qquad
    \alpha=-0.5.
\end{equation*}
The cubic term gives a nonlinear restoring force centered at $a(t)$, while the mean-field term couples each coordinate to the empirical mean $\bar x$. The initial density is a Gaussian density. On the computational box, we evaluate
it using the wrapped displacement
\begin{equation*}
    \delta_i(x,\mu)
    :=
    \left(x_i-\mu_i+\frac{L}{2}\right)\bmod L-\frac{L}{2},
    \qquad
    L=4,
\end{equation*}
and set
\begin{equation*}
    u^0(x)
    =
    (2\pi\sigma^2)^{-d/2}
    \exp\left(
    -\frac{1}{2\sigma^2}
    \norm{\delta(x,\mu^0)}_2^2
    \right),
    \qquad
    \sigma^2=0.1.
\end{equation*}
The initial mean is placed along a line in the coordinate index,
\begin{equation*}
    \mu_i^0
    =
    1.46+1.27\frac{i-1}{d-1},
    \qquad
    i=1,\ldots,d.
\end{equation*}

Although the equation is posed on $\R^d$, the computation is carried out on the box $[0,4)^d$. The box is chosen large enough to contain all probability mass up to double precision over the time interval considered.

We note that the right-hand side of \eqref{eq:NumExp:FokkerPlanck} depends on time so that the stated Fokker--Planck equation defines a non-autonomous evolution problem. Although the well-posedness and a posteriori error analysis developed above are formulated for autonomous systems, the numerical results below indicate that the DFI scheme exhibits analogous behavior in this non-autonomous setting.

\subsubsection{Nonlinear parametrization with neural network}
Since the unknown is a probability density, we enforce positivity directly in the parametrization by setting
\begin{equation*}
    \Phi(\theta)(x)=\exp(-\varphi(\theta)(x)).
\end{equation*}
Here $\varphi(\theta)$ is a periodic fully connected feedforward neural network; see \Cref{sec:NumExp:AC:DNNSetup}. This makes the neural representation periodic on the computational box $[0,4)^d$. The network $\varphi(\theta)$ has four hidden layers, each of width $64$, and
uses the swish activation applied componentwise. The trainable parameters are the entries of the weight matrices, biases, output weights, and output bias. In dimension $d=10$, this architecture gives $p = 18625$ trainable parameters.

\subsubsection{Setup of the DFI scheme}
The empirical least-squares problem is formed from importance samples on the computational box. At each time step, we draw
\begin{equation*}
    \mathcal X_k=\{x_\ell\}_{\ell=1}^{N_s}\subset[0,4)^d,
    \qquad
    N_s=2000,
\end{equation*}
from an adaptive mixture proposal. The proposal consists of $1500$ samples from a Gaussian fitted to the current approximation $\hat u_k$ and $500$ samples from the uniform distribution on $[0,4)^d$. Let $q_k$ denote this proposal density. The empirical norm used in the least-squares problem is the importance-sampling approximation of the
$L^2([0,4)^d)$ norm,
\begin{equation}\label{eq:NumExp:FEP:EmpNorm}
    \norm{w}_{2,\mathcal X_k}^2
    =
    \frac{1}{N_s}
    \sum_{\ell=1}^{N_s}
    \frac{\abs{w(x_\ell)}^2}{q_k(x_\ell)}.
\end{equation}
The semi-implicit Euler DFI update is then conducted as in Algorithm~\ref{alg:DFI} with the empirical norm \eqref{eq:NumExp:FEP:EmpNorm} and the Gaussian fitted to the current solution $\hat{u}_k$ at each time step.
We use a time step size of $10^{-3}$ and keep $\beta$ and $\eta$ fixed over all time
steps.
We initialize the parameter variable to $\theta_0$ computed by fitting the initial condition with the Adam optimizer for $2 \times 10^{5}$ full-batch iterations. To ensure our results are not polluted by the initial error, we perform this fit using a finer sample of size $2 \times 10^{5}$, half being drawn from the Gaussian initial condition itself and the other half from the uniform measure, and using importance sampling to estimate the $L^2$-error. The velocity variable is initialized to $v_0=0$.

We compare DFI and Tikhonov-DF using moment diagnostics. The reference moments $\mu_k^{\mathrm{ref}}$ and $\Sigma_k^{\mathrm{ref}}$ are computed from the $10^5$ reference particles from the SDE \eqref{eq:fpe_sde} with the Euler--Maruyama method and time step $10^{-4}$. The moments $\widehat\mu_k$ and $\widehat\Sigma_k$ are computed by importance sampling from the same adaptive proposal used for evaluation. The weights are self-normalized using the ratio between the approximate density $\hat u_k$ and the proposal density $q_k$. Because the computation is carried out on a periodic box, the mean is estimated
coordinate-wise by circular averaging after mapping each coordinate from
$[0,4)$ to the unit circle. The covariance is then computed from wrapped
displacements around this circular mean. We report the pointwise relative errors
\begin{equation}
    e_{\mu,k}
    =
    \frac{
        \norm{\widehat\mu_k-\mu_k^{\mathrm{ref}}}_2
    }{
        \norm{\mu_k^{\mathrm{ref}}}_2
    },
    \qquad
    e_{\Sigma,k}
    =
    \frac{
        \norm{
            \operatorname{diag}(\widehat\Sigma_k)
            -
            \operatorname{diag}(\Sigma_k^{\mathrm{ref}})
        }_2
    }{
        \norm{
            \operatorname{diag}(\Sigma_k^{\mathrm{ref}})
        }_2
    }\,,
    \label{eq:NumExp:FPEPointwiseMomentErrors}
\end{equation}
and the integrated moment errors,
\begin{equation*}
    E_\mu
    =
    \left(
    \frac{
        \sum_{k=0}^{N_t-1}h\,
        \norm{\widehat\mu_k-\mu_k^{\mathrm{ref}}}_2^2
    }{
        \sum_{k=0}^{N_t-1}h\,
        \norm{\mu_k^{\mathrm{ref}}}_2^2
    }
    \right)^{1/2},
\end{equation*}
and
\begin{equation*}
    E_\Sigma
    =
    \left(
    \frac{
        \sum_{k=0}^{N_t-1}h\,
        \norm{
            \operatorname{diag}(\widehat\Sigma_k)
            -
            \operatorname{diag}(\Sigma_k^{\mathrm{ref}})
        }_2^2
    }{
        \sum_{k=0}^{N_t-1}h\,
        \norm{
            \operatorname{diag}(\Sigma_k^{\mathrm{ref}})
        }_2^2
    }
    \right)^{1/2}.
\end{equation*}

\subsubsection{Results}
We use this high-dimensional Fokker--Planck example to demonstrate that the inertial memory in DFI improves robustness in a more challenging nonlinear parametrization. For both DFI and Tikhonov-DF, the regularization parameters are tuned for best performance with respect to the mean error via grid search, which is analogous to the procedure used in \Cref{sec:numexp_allen_cahn}. In DFI, this means tuning both the memory parameter $\beta$ and the regularization strength $\eta^2$. In Tikhonov-DF, we tune the corresponding Tikhonov regularization parameter.

At each time step, the residual is evaluated at $N_s=2000$ sample points from the adaptive mixture proposal. Thus the instantaneous Dirac--Frenkel signal is randomized. The results in Figure~\ref{fig:fokker_planck_interror} show that DFI gives a more robust time evolution of the moment diagnostics in this regime. For both the mean and the diagonal covariance, Tikhonov-DF initially follows the reference dynamics, but the error grows sharply after a short time. DFI avoids this error growth and maintains substantially smaller moment errors over the time interval.
Furthermore, Figure \ref{fig:fokker_planck_interror_vs_sketch} indicates that the DFI dynamics are numerically less sensitive to the reduced local information in small sketches than Tikhonov-DF, consistently with the behavior observed in the Allen--Cahn example. Again, this robustness can be attributed to the fact that the corresponding least-squares formulation depends less strongly on localized information because of the memory term.
Overall, these results agree with the trends previously observed for the Allen--Cahn problem.

\begin{figure}[t]
    \centering
    \includegraphics[width=1.0\textwidth]{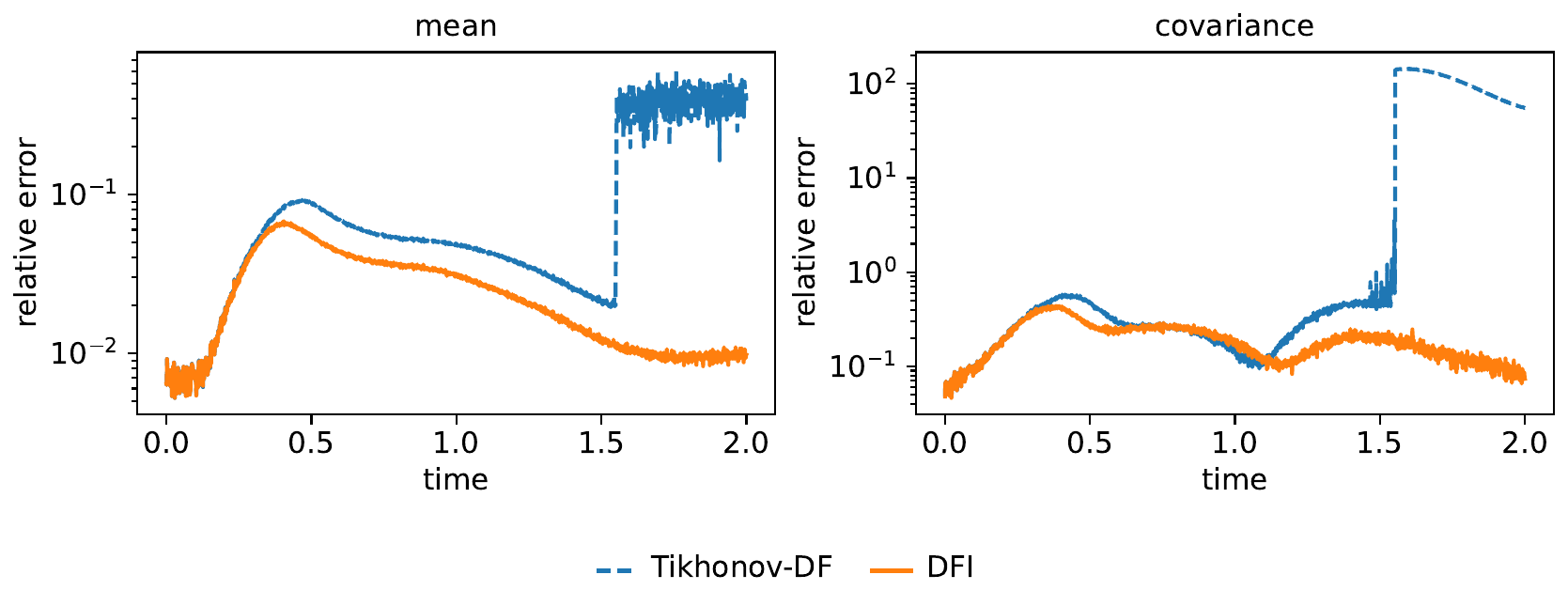}
    \caption{Fokker--Planck in 10D: DFI is more robust in the sample-based high-dimensional Fokker--Planck experiment. The instantaneous least-squares problems are formed using only $N_s=2000$ sample points. The regularization parameters are set to $\eta^2=10^6$ and $\beta = 0.09$. Tikhonov-DF develops large pointwise errors in the mean and diagonal covariance after a short time, whereas DFI maintains smaller errors over the time interval. Curves are averaged over five independent runs.}
    \label{fig:fokker_planck_interror}
\end{figure}

\begin{figure}[t]
    \centering
    \includegraphics[width=1.0\textwidth]{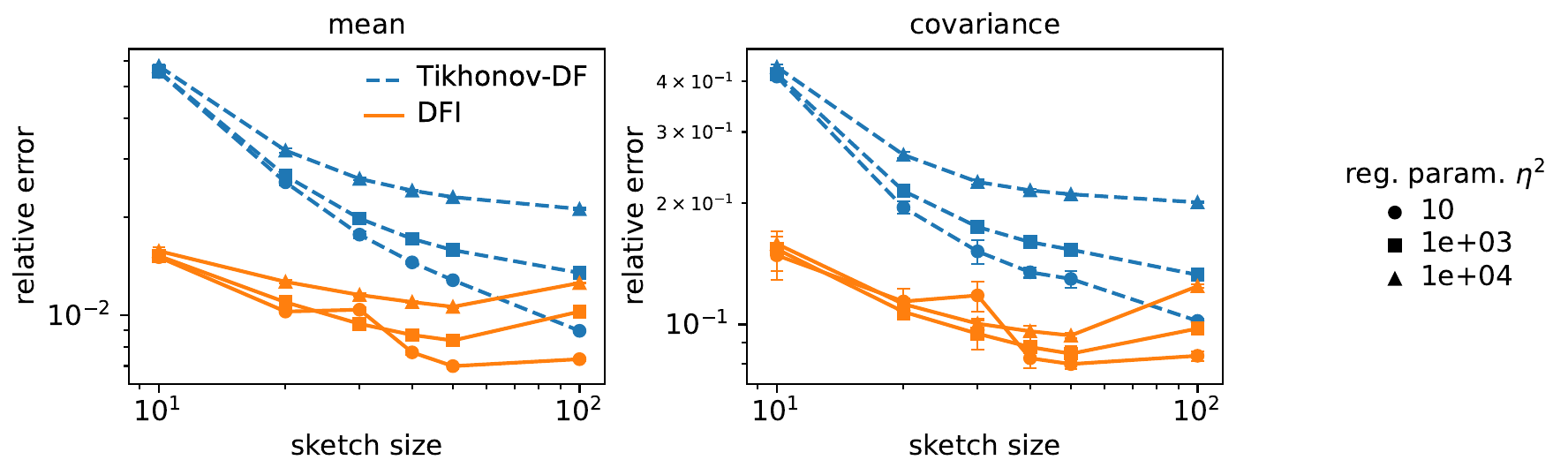}
    \caption{Fokker--Planck in 10D: DFI is more robust to smaller sketch size due to the memory term supplementing information from past solves. Results are averaged over five independent runs.}
    \label{fig:fokker_planck_interror_vs_sketch}
\end{figure}

\section{Conclusions}
By adding inertia to the Dirac--Frenkel dynamics, the proposed DFI scheme allows parameter velocity information from the past trajectory to persist in directions that are weakly informed by the instantaneous Jacobian. This mechanism is useful precisely in the regimes where the current Jacobian does not provide enough reliable information about all parameter velocity directions. Such situations occur when the parametrization is redundant, when the Jacobian is ill-conditioned or nearly rank deficient, when the least-squares problem is strongly regularized, or when the Jacobian is sketched. In these cases, Tikhonov-regularized Dirac--Frenkel dynamics suppress weakly informed directions instantaneously. DFI instead allows velocity in such directions to persist, while still using the current Dirac--Frenkel residual to correct directions that are well resolved by the Jacobian. 

We established well-posedness of the continuous DFI system and introduced a semi-implicit Euler discretization, which requires one regularized least-squares solve per time step, with the previous velocity appearing as an anchor. \tc{For the time-discrete scheme, we derived a posteriori bounds in which the state error is controlled by augmented defects combining the residual, a Tikhonov velocity term, and an inertial velocity-increment term, under a defect-control restriction at the Tikhonov scale $\varepsilon^2$. 
Our analysis suggests that, when the residual remains small and the DFI velocities are temporally coherent, the effective regularization $\eta$ can be chosen larger than $\varepsilon$ while retaining accuracy of the integration scheme, potentially reducing the cost of the instantaneous corrections.
The provided numerical experiments on the Allen-Cahn and Fokker-Planck equations demonstrate that DFI can remain accurate with stronger regularization and with more aggressively sketched least-squares solves.}

\tc{ 
\section*{Acknowledgments}
The authors would like to thank Fabio Nobile for helpful discussions leading to the improvement of the a posteriori bounds of the time-discrete scheme and the manuscript as a whole.
}

\bibliography{biblio}

\end{document}